\documentclass[11pt]{amsart}
\usepackage{amsmath,amsfonts,amssymb,amsthm,mathtools,mathrsfs,bm}
\usepackage{booktabs,enumitem}
\usepackage[T1]{fontenc}
\usepackage[colorlinks=true,citecolor=blue,linkcolor=blue,urlcolor=blue]{hyperref}
\allowdisplaybreaks
\newcommand{\amal}[5]{
  {#1}^{\,#2}_{\vphantom{#3}}\!\times_{#3}^{#4}\!{#5}
}

\newcommand{\R}{\mathbb R}

\newcommand{\Z}{\mathbb Z}
\newcommand{\HH}{\mathscr H}
\newcommand{\DD}{\mathscr D}
\newcommand{\GG}{\mathfrak G}
\newcommand{\Id}{\operatorname{Id}}
\newcommand{\spc}{\operatorname{spec}}
\newcommand{\Ker}{\operatorname{Ker}}
\newcommand{\Iso}{\operatorname{Iso}}

\newcommand{\graddeg}{\operatorname{deg}^{\nabla}_{\GG}}
\newcommand{\rel}{\operatorname{deg}^{\nabla}_{\GG,\mathrm{rel}}}
\newcommand{\Dg}{\mathfrak d_{\GG}}
\newcommand{\iso}{\simeq_{\GG}}
\newcommand{\du}{\mathrm d}
\newcommand{\pair}[2]{\left\langle #1,#2\right\rangle}

\newcommand{\starprod}{\mathbin{\star}}

\newcommand{\calC}{\mathcal C}

\newcommand{\Eminus}[1]{E^-(#1)}

\newtheorem{theorem}{Theorem}[section]
\newtheorem{proposition}[theorem]{Proposition}
\newtheorem{lemma}[theorem]{Lemma}
\newtheorem{corollary}[theorem]{Corollary}
\newtheorem{definition}[theorem]{Definition}
\newtheorem{assumption}[theorem]{Assumption}
\newtheorem{remark}[theorem]{Remark}
\newtheorem{example}[theorem]{Example}

\numberwithin{equation}{section}

\begin{document}

\title{Equivariant Gradient Degree Comparison via Clarke Duality for Strongly Indefinite Hamiltonian Systems}

\author{Shi Yu}
\address{Department of Mathematics, Odessa College,
201 W University Blvd, Odessa, TX 79764, USA}
\email{syu@odessa.edu}

\begin{abstract}
We study equivariant gradient degree comparisons for
strongly indefinite Hamiltonian systems through Clarke
duality. Under suitable convexity and compactness
assumptions, we establish a correspondence between the
primal and dual equations that preserves the symmetries
of their solutions. We derive an identity relating their
negative spectral representations and show that the
Galerkin-normalized local index of the primal linearization
coincides with the equivariant gradient degree determined
by the dual Hessian. This comparison extends to
infinite-dimensional settings through stabilized Galerkin
approximations. We also establish relations between
primal and dual crossing forms, equivariant spectral
flow, and local equivariant gradient degrees at
nondegenerate fixed solutions. The results provide a
framework for studying equivariant bifurcation in
strongly indefinite systems. As an application, we
consider periodic Hamiltonian ring equations with
dihedral symmetry and illustrate the computation of
equivariant gradient degrees and bifurcation invariants.
\end{abstract}

\subjclass[2020]{Primary 47H11, 58E09, 47J15, 37J46}

\keywords{Clarke duality, equivariant gradient degree,
strongly indefinite Hamiltonian systems,
Galerkin approximation, negative spectral representations,
equivariant spectral flow, bifurcation, dihedral symmetry}

\maketitle

\section{Introduction}

Variational methods for Hamiltonian systems often lead
to strongly indefinite functionals whose associated
linear operators possess infinitely many positive and
negative spectral directions. This spectral structure
presents difficulties for the direct application of
classical Morse theory and topological degree methods.
An important approach to such problems is provided by
Clarke duality, which transforms a Hamiltonian variational
problem into a dual formulation through convex conjugation.
The dual variational method was developed in the study
of periodic Hamiltonian trajectories and has become a
useful tool for analyzing convex Hamiltonian systems;
see \cite{Clarke1979,ClarkeEkeland1980,Clarke1985}.
Related variational frameworks for Hamiltonian elliptic
systems are discussed in \cite{Bonheure2014}.

When a Hamiltonian system admits a group of symmetries,
the corresponding variational problem can be studied
using equivariant topological methods. Equivariant
gradient degree extends classical topological degree
by incorporating information about the symmetries of
solutions. Its construction, computational properties,
and applications to symmetric differential equations
have been developed in
\cite{Geba1997,RybickiGrad,AED,SURVEY,RuanRybicki2008}.
For variational problems with $SO(2)$ symmetry,
related bifurcation methods are developed in \cite{Rybi1}.

The application of equivariant gradient degree to
strongly indefinite functionals requires additional
considerations because the negative spectral subspaces
may be infinite-dimensional. Go\l\k{e}biewska and
Rybicki \cite{GolebiewskaRybicki2011} developed an
equivariant degree theory for strongly indefinite
invariant functionals and established global bifurcation
results for critical orbits. Related approaches to
bifurcation through spectral flow were developed by
Fitzpatrick, Pejsachowicz, and Recht
\cite{FitzpatrickPejsachowiczRecht1999}, while the
spectral-flow framework and its relation to crossing
forms were studied by Robbin and Salamon
\cite{RobbinSalamon1995}.

Although Clarke duality establishes a correspondence
between solutions of the primal and dual equations,
it does not immediately imply an equality of their
equivariant gradient degrees. The linearizations of
the two equations may have different negative spectral
representations, and the ordinary finite-Morse-index
equivariant gradient degree is not directly applicable
to a strongly indefinite primal operator. This raises
the question of how the equivariant gradient degrees
of the primal and dual linearizations are related
and whether a suitable normalization of finite-dimensional
approximations produces a stable invariant.

The purpose of this paper is to establish such a
comparison within an abstract equivariant Clarke
duality framework. We first prove an identity relating
the negative spectral representations of the primal
and dual Hessians. This identity yields a finite-dimensional
equivariant gradient degree comparison and identifies
the contribution of the reference operator. We then
construct a Galerkin-normalized local index for strongly
indefinite primal linearizations and prove that,
under suitable compactness and positivity assumptions,
this index stabilizes and agrees with the equivariant
gradient degree determined by the negative spectral
representation of the dual Hessian. We also establish
a correspondence between primal and dual crossing
forms and obtain the corresponding relation between
their equivariant spectral flows.

For nonlinear problems, we derive a local additive
comparison formula for equivariant gradient degrees
at isolated nondegenerate fixed solutions. We distinguish
this result from the ordinary equivariant compact-map
degree comparison for finite symmetry groups, which
follows from the classical commutativity property
of the fixed-point index; see
\cite{Krasnoselskii,GranasDugundji2003}.

We further consider bifurcation from the trivial
branch of equivariant Hamiltonian problems. Changes
in the equivariant gradient degree of the dual
linearization provide a means of detecting bifurcation.
The primal--dual correspondence transfers
the resulting solutions and preserves their isotropy
groups. The bifurcation analysis is related to the
classical global bifurcation theory of Rabinowitz
\cite{Rabinowitz} and to equivariant methods for
symmetric Hamiltonian and periodic systems
\cite{Dancer,SY1}. Applications of related symmetry-based
bifurcation methods include periodic solutions in the
restricted $n$-body problem \cite{Az1}, polygonal relative
equilibria \cite{Az2}, spiral waves in the complex
Ginzburg--Landau equation \cite{GA_Ghanem_Krawcewicz_CGL},
and periodic Weyl and Riesz fractional equations
\cite{YuFractional}.

As an application, we study periodic solutions of
coupled Hamiltonian equations arranged in a ring
with dihedral symmetry. Equivariant degree methods
have also been applied to periodic solutions in delayed
rings of identical cells \cite{YuRingDelay} and to
inventory oscillations in ring supply-chain networks
\cite{YuSupplyChain}. Here we focus instead on the
primal--dual gradient-degree comparison for a
Hamiltonian ring. Fourier decomposition reduces the
linear comparison to finite-dimensional equivariant
representations. A five-site example illustrates the
computation of the dual negative spectral representations,
the corresponding equivariant gradient degree, and
a local bifurcation invariant.

The paper is organized as follows. Section 2 establishes
the abstract equivariant Clarke duality correspondence.
Section 3 examines the Hessians, negative spectral
representations, and crossing forms. Section 4 develops
the finite-dimensional equivariant gradient degree
comparison. Section 5 establishes the infinite-dimensional
normalized linear index. Section 6 discusses equivariant
bifurcation, and Section 7 applies the results to
periodic Hamiltonian ring systems. The appendices
provide the basic equivariant gradient degree calculations
and subgroup identifications used in the example.

\section{Abstract equivariant Clarke duality}\label{sec:framework}

\begin{assumption}
\label{ass:convex}
Let $\GG$ be a compact Lie group and $\HH$ a real
separable Hilbert space carrying a strongly continuous
orthogonal representation of $\GG$.

Assume that
\[
A:\DD(A)\subset\HH\longrightarrow\HH
\]
is a densely defined, self-adjoint, invertible,
$\GG$-equivariant operator with compact inverse
\[
A^{-1}:\HH\longrightarrow\HH.
\]

Let $\mathcal H:\HH\to\R$ be a convex,
continuously Fr\'echet differentiable,
$\GG$-invariant functional satisfying
\[
\mathcal H(0)=0,\qquad
\nabla\mathcal H(0)=0,
\]
and assume that
\[
\nabla\mathcal H:\HH\longrightarrow\HH
\]
is a $\GG$-equivariant homeomorphism.

When Hessian formulas at zero are required,
assume additionally that $\nabla\mathcal H$
is Fr\'echet differentiable at zero and
\begin{equation}
\label{eq:B}
B:=D(\nabla\mathcal H)(0),\qquad
b_0\Id\le B\le b_1\Id,
\qquad b_0,b_1>0.
\end{equation}
\end{assumption}

Since $A$ is invertible with compact inverse,
it has compact resolvent, and its spectrum
may be unbounded in both directions.

Define the convex conjugate of $\mathcal H$ by
\[
\mathcal H^*(u):=
\sup_{z\in\HH}\left\{
\langle u,z\rangle-\mathcal H(z)\right\}.
\]

\begin{lemma}
\label{lem:dual-gradient}
Under Assumption~\ref{ass:convex},
the convex conjugate $\mathcal H^*$
belongs to $C^1(\HH,\R)$ and satisfies
\begin{equation}
\label{eq:dual-gradient}
\nabla\mathcal H^*(u)=(\nabla\mathcal H)^{-1}(u), \qquad u\in\HH.
\end{equation}
Moreover, $\mathcal H^*$ is $\GG$-invariant, and $\nabla\mathcal H^*$ is a $\GG$-equivariant homeomorphism.
\end{lemma}

\begin{proof}
Set $N=\nabla\mathcal H$. Since $N$ is a homeomorphism, for every $u\in\HH$ there
exists a unique
\[
z(u)=N^{-1}(u)
\]
such that
\[
\nabla\mathcal H(z(u))=u.
\]

By convexity and differentiability of
$\mathcal H$, for every $z\in\HH$,
\[
\begin{aligned}
\mathcal H(z)
&\ge \mathcal H(z(u))
+\langle\nabla\mathcal H(z(u)), z-z(u)\rangle\\
&=
\mathcal H(z(u)) +\langle u,z-z(u)\rangle.
\end{aligned}
\]
Rearranging gives
\[
\langle u,z\rangle-\mathcal H(z)
\le \langle u,z(u)\rangle-\mathcal H(z(u)).
\]
Since this inequality holds for every
$z\in\HH$, and equality holds at
$z=z(u)$, the supremum defining
$\mathcal H^*(u)$ is attained at $z(u)$.
Consequently,
\[
\mathcal H^*(u)=\langle u,z(u)\rangle-\mathcal H(z(u)).
\]

We next prove that $\mathcal H^*$ is
continuously Fr\'echet differentiable.
For any $h\in\HH$, the definition of
the convex conjugate gives
\[
\begin{aligned}
\mathcal H^*(u+h)&=
\sup_{z\in\HH}\{\langle u+h,z\rangle-\mathcal H(z)\}\\
&\ge\langle u+h,z(u)\rangle-\mathcal H(z(u))\\
&=
\langle u,z(u)\rangle-\mathcal H(z(u))
+\langle h,z(u)\rangle\\
&=\mathcal H^*(u)+\langle h,z(u)\rangle.
\end{aligned}
\]

Similarly, choosing $z=z(u+h)$ in
the definition of $\mathcal H^*(u)$,
we obtain
\[
\begin{aligned}
\mathcal H^*(u)&=
\sup_{z\in\HH}\{\langle u,z\rangle-\mathcal H(z)\}\\
&\ge
\langle u,z(u+h)\rangle
-\mathcal H(z(u+h))\\&=
\langle u+h,z(u+h)\rangle
-\mathcal H(z(u+h))\\
&\qquad-\langle h,z(u+h)\rangle\\
&=\mathcal H^*(u+h)
-\langle h,z(u+h)\rangle.
\end{aligned}
\]

Combining these inequalities yields
\[
\begin{aligned}
0&\le
\mathcal H^*(u+h)-\mathcal H^*(u)-\langle h,z(u)\rangle\\
&\le\langle h,z(u+h)-z(u)\rangle\\
&\le\|h\|\,\|z(u+h)-z(u)\|.
\end{aligned}
\]

For $h\ne0$, it follows that
\[
\begin{aligned}
0&\le
\frac{\mathcal H^*(u+h)-\mathcal H^*(u)
-\langle h,z(u)\rangle}{\|h\|}\\
&\le\|z(u+h)-z(u)\|.
\end{aligned}
\]
Since $z=N^{-1}$ is continuous,
\[
\lim_{h\to0}\|z(u+h)-z(u)\|=0.
\]
Therefore,
\[
\lim_{h\to0}\frac{\mathcal H^*(u+h)-\mathcal H^*(u)
-\langle h,z(u)\rangle}{\|h\|}=0.
\]
Thus $\mathcal H^*$ is Fr\'echet
differentiable at $u$, with
\[
D\mathcal H^*(u)[h]=\langle z(u),h\rangle.
\]
By the Riesz representation theorem,
\[
\nabla\mathcal H^*(u)=z(u)
=(\nabla\mathcal H)^{-1}(u).
\]
Since $(\nabla\mathcal H)^{-1}$ is continuous, we conclude that
\[
\mathcal H^*\in C^1(\HH,\R).
\]

Finally, for every $g\in\GG$, the orthogonality of the group action
and the $\GG$-invariance of $\mathcal H$ imply
\[
\begin{aligned}
\mathcal H^*(gu)&=
\sup_{z\in\HH}\{\langle gu,z\rangle-\mathcal H(z)\}\\
&=\sup_{z\in\HH}
\{\langle u,g^{-1}z\rangle-\mathcal H(g^{-1}z)\}\\
&=\mathcal H^*(u).
\end{aligned}
\]
Hence $\mathcal H^*$ is $\GG$-invariant.

Moreover, the equivariance of $N$ implies that of $N^{-1}$, so
\[
\begin{aligned}
\nabla\mathcal H^*(gu)
&=N^{-1}(gu)\\
&=gN^{-1}(u)\\
&=g\nabla\mathcal H^*(u).
\end{aligned}
\]
Thus $\nabla\mathcal H^*$ is a
$\GG$-equivariant homeomorphism.
\end{proof}

The original equation is
\begin{equation}
\label{eq:primal}
F(z):=Az+\nabla\mathcal H(z)=0,
\qquad z\in\DD(A),
\end{equation}
with the associated primal action
\[
\mathcal A(z)=\frac12\langle Az,z\rangle+\mathcal H(z).
\]

\begin{theorem}[Equivariant primal--dual correspondence]\label{thm:correspondence}
Under Assumption~\ref{ass:convex}, the equation \eqref{eq:primal} is equivalent to
\[
 F^*(u):=A^{-1}u+\nabla\mathcal H^*(u)=0.
\]
Corresponding solutions are related by
\begin{equation}
\label{eq:correspondence}
 u=\nabla\mathcal H(z),\qquad z=\nabla\mathcal H^*(u)=-A^{-1}u.
\end{equation}
The correspondence is $\GG$-equivariant and
\[
\Iso_{\GG}(z)=\Iso_{\GG}(u).
\]
The dual equation is the Euler equation of
\[
 \Psi(u)=\frac12\pair{A^{-1}u}{u}+\mathcal H^*(u).
\]
Moreover, $\Psi\in C^1(\HH,\R)$ and
\[
\nabla\Psi(u)=F^*(u).
\]
\end{theorem}

\begin{proof}
Set $N=\nabla\mathcal H$ and $N^{-1}=\nabla\mathcal H^*$. For $u=N(z)$,
\[
\begin{aligned}
 F(z)=0&\iff Az+u=0
 \iff z=-A^{-1}u\\&\iff N^{-1}(u)+A^{-1}u=0
 \iff F^*(u)=0.
\end{aligned}
\]
Conversely, $F^*(u)=0$ implies
\[
 z=N^{-1}(u)=-A^{-1}u\in\DD(A),\qquad
 F(z)=-u+N(N^{-1}(u))=0.
\]
Moreover, for $z\in\DD(A)$,
\[
\begin{aligned}
F^*(N(z))
 &=A^{-1}N(z)+N^{-1}(N(z))\\
 &=A^{-1}(N(z)+Az)=A^{-1}F(z).
\end{aligned}
\]
For $g\in\GG$ and $v\in\HH$,
\[
\begin{aligned}
\langle N(gz),v\rangle
 &=D\mathcal H(gz)[v]
 =D\mathcal H(z)[g^{-1}v]\\
 &=\langle N(z),g^{-1}v\rangle
 =\langle gN(z),v\rangle.
\end{aligned}
\]
Thus $N(gz)=gN(z)$ and $N^{-1}(gu)=gN^{-1}(u)$, whence
\[
 gz=z\iff N(gz)=N(z)\iff gu=u,
 \qquad \Iso_{\GG}(z)=\Iso_{\GG}(u).
\]
Finally, self-adjointness of $A^{-1}$ gives
\[
\begin{aligned}
D\Psi(u)[v]
 &=\tfrac12\langle A^{-1}v,u\rangle
  +\tfrac12\langle A^{-1}u,v\rangle
  +D\mathcal H^*(u)[v]\\
 &=\langle A^{-1}u+\nabla\mathcal H^*(u),v\rangle.
\end{aligned}
\]
Consequently,
\[
 \nabla\Psi(u)=F^*(u),\qquad
 D\Psi(u)=0\iff F^*(u)=0.
\]
\end{proof}

\begin{remark}\label{rem:intertwine}
Wherever the two sides are defined,
\begin{equation}
\label{eq:intertwine}
 F^*(\nabla\mathcal H(z))=A^{-1}F(z).
\end{equation}
The identity identifies the zero sets. It does not imply equality of gradient degrees, since $A^{-1}$ may reverse orientation on negative subrepresentations.
\end{remark}

For every closed subgroup $K\le\GG$, define
the $K$-fixed-point subspace by
\[
\HH^K
:=
\{z\in\HH:gz=z\text{ for all }g\in K\}.
\]
Since the action of $\GG$ on $\HH$ is orthogonal, $\HH^K$ is a closed linear subspace of $\HH$.

\begin{theorem}[Restriction to fixed-point equations]
\label{thm:fixed-dual}
Under Assumption~\ref{ass:convex}, for every closed subgroup $K\le\GG$, the Legendre transformation restricts to a homeomorphism
\[
\nabla\mathcal H:\HH^K\longrightarrow\HH^K
\]
and establishes a bijection between the
zeros of the primal and dual equations
on $\HH^K$.
\end{theorem}

\begin{proof}
Set
\[
N=\nabla\mathcal H,
\qquad N^{-1}=\nabla\mathcal H^*.
\]
For $z\in\HH^K$ and $k\in K$,
equivariance gives
\[
kN(z)=N(kz)=N(z).
\]
Similarly, for $u\in\HH^K$,
\[
kN^{-1}(u)=N^{-1}(ku)=N^{-1}(u).
\]
Consequently,
\[
N(\HH^K)=\HH^K,
\qquad
N^{-1}(\HH^K)=\HH^K.
\]
Thus $N|_{\HH^K}$ is a homeomorphism
with inverse $N^{-1}|_{\HH^K}$.

Since $A^{-1}$ is $\GG$-equivariant,
\[
A^{-1}(\HH^K)\subset\HH^K.
\]
By \eqref{eq:intertwine},
\[
F^*(N(z))=A^{-1}F(z),
\qquad z\in\DD(A)\cap\HH^K.
\]
Since $A^{-1}$ is injective,
\[
F(z)=0
\quad\Longleftrightarrow\quad
F^*(N(z))=0.
\]
Conversely, every dual zero $u\in\HH^K$
corresponds to
\[
z=N^{-1}(u)=-A^{-1}u
\in\DD(A)\cap\HH^K.
\]
Hence the primal and dual zeros on
$\HH^K$ correspond bijectively.
\end{proof}

\section{Hessians and negative equivariant representations}\label{sec:hess}

Under condition \eqref{eq:B}, the primal and dual Hessians at zero are given by
\[
P:=A+B,\qquad Q:=A^{-1}+B^{-1}.
\]
The first is unbounded with the same domain as $A$, and the second is a bounded self-adjoint operator. As $B^{-1}\ge b_1^{-1}\Id$ and $A^{-1}$ is compact, $Q$ has only finitely many negative eigenvalues, counting multiplicities, whenever $0\notin\spc(Q)$.

\begin{lemma}
\label{lem:kernel}
Under Assumption~\ref{ass:convex}, let
$z\in\DD(A)$ and $u=\nabla\mathcal H(z)$
be corresponding solutions of the primal
and dual equations.

Suppose that $\mathcal H$ is $C^2$ in a neighborhood of $z$ and that
\[
B_z:=D(\nabla\mathcal H)(z)
\ge c\Id>0.
\]
Define
\[
P_z:=A+B_z,
\qquad
Q_u:=A^{-1}+B_z^{-1}.
\]
Then
\[
D(\nabla\mathcal H^*)(u)=B_z^{-1},
\]
and
\[
\Ker P_z\simeq_L\Ker Q_u,
\qquad
L=\Iso_{\GG}(z)=\Iso_{\GG}(u).
\]
In particular, at the trivial solution,
\[
P=A+B,\qquad Q=A^{-1}+B^{-1},
\]
and
\[
\Ker P\simeq_{\GG}\Ker Q.
\]
\end{lemma}

\begin{proof}
By the inverse function theorem and
\eqref{eq:dual-gradient},
\[
D(\nabla\mathcal H^*)(u)=[D(\nabla\mathcal H)(z)]^{-1}=B_z^{-1}.
\]
Thus the primal and dual linearizations are respectively
\[
P_z=A+B_z,
\qquad
Q_u=A^{-1}+B_z^{-1}.
\]

For every $h\in\DD(A)$,
\[
\begin{aligned}
Q_uB_zh
&=(A^{-1}+B_z^{-1})B_zh\\
&=A^{-1}B_zh+h\\
&=A^{-1}(A+B_z)h\\
&=A^{-1}P_zh.
\end{aligned}
\]
Consequently,
\[
P_zh=0
\quad\Longrightarrow\quad
Q_u(B_zh)=0.
\]

Conversely, if $Q_uv=0$, then
\[
B_z^{-1}v=-A^{-1}v\in\DD(A).
\]
Hence
\[
P_z(B_z^{-1}v)=0.
\]
Therefore,
\[
B_z:\Ker P_z\longrightarrow\Ker Q_u
\]
is an isomorphism with inverse
$B_z^{-1}$.

Let
\[
L=\Iso_{\GG}(z)=\Iso_{\GG}(u).
\]
Differentiating
\[
\nabla\mathcal H(gz)
=g\nabla\mathcal H(z),
\qquad g\in L,
\]
gives
\[
B_zg=gB_z.
\]
Thus the kernel isomorphism is
$L$-equivariant, and
\[
\Ker P_z\simeq_L\Ker Q_u.
\]

At $z=u=0$, we have
\[
B_0=B,\qquad
P=A+B,\qquad
Q=A^{-1}+B^{-1}.
\]
Since $0$ is fixed by $\GG$, the
kernel isomorphism is $\GG$-equivariant.
\end{proof}

Let $RO(\GG)$ denote the real representation ring. For an invertible self-adjoint, finite-dimensional $\GG$-equivariant operator $T$, let $\Eminus{T}$ be its negative spectral $\GG$-representation. Its class is written $[\Eminus{T}]\in RO(\GG)$.

\begin{theorem}\label{thm:inertia}
Let $V$ be a finite-dimensional orthogonal
$\GG$-representation. Suppose $A=A^*$ and
$B>0$ are $\GG$-equivariant, and that
$A$ and $A+B$ are invertible. Set
\[
P=A+B,\qquad Q=A^{-1}+B^{-1}.
\]
Then
\begin{equation}
\label{eq:ROidentity}
[\Eminus{A}]
=[\Eminus{P}]+[\Eminus{Q}] \quad\text{in }RO(\GG).
\end{equation}
\end{theorem}

\begin{proof}
Set
\[
T=B^{-1/2}AB^{-1/2}.
\]
Then
\[
\begin{aligned}
A&=B^{1/2}TB^{1/2},\\
P&=B^{1/2}(T+\Id)B^{1/2},\\
Q&=B^{-1/2}(T^{-1}+\Id)B^{-1/2}.
\end{aligned}
\]

Positive $\GG$-equivariant congruence preserves the isomorphism class of the negative spectral representation.

Consequently,
\[
\begin{aligned}
\Eminus{A}&\iso\Eminus{T},\\
\Eminus{P}&\iso\Eminus{T+\Id},\\
\Eminus{Q}&\iso\Eminus{T^{-1}+\Id}.
\end{aligned}
\]

Since $A$ and $P$ are invertible,
\[
0,-1\notin\spc(T).
\]
Define the $\GG$-invariant spectral subspaces
\[
\begin{aligned}
V_1
&=\bigoplus_{\lambda<-1}\Ker(T-\lambda\Id),\\
V_2 &=\bigoplus_{-1<\lambda<0}\Ker(T-\lambda\Id),
\end{aligned}
\]
where the sums are over the eigenvalues
of $T$.
Hence
\[
\begin{aligned}
\Eminus{T}&=V_1\oplus V_2,\\
\Eminus{T+\Id}&=V_1,\\
\Eminus{T^{-1}+\Id}&=V_2.
\end{aligned}
\]

Therefore,
\[
\begin{aligned}
[\Eminus{A}]
&=[\Eminus{T}]\\
&=[V_1]+[V_2]\\
&=[\Eminus{P}]+[\Eminus{Q}]
\end{aligned}
\]
in $RO(\GG)$.
\end{proof}

\begin{theorem}\label{thm:crossing}
Let $\lambda\mapsto A_\lambda$ be a $C^1$ family of invertible self-adjoint equivariant operators with a common domain and differentiable inverses. Let $\lambda\mapsto B_\lambda$ be a norm-$C^1$ family of positive equivariant isomorphisms, and set
\[
P_\lambda=A_\lambda+B_\lambda,\qquad  Q_\lambda=A_\lambda^{-1}+B_\lambda^{-1}.
\]
At any $\lambda_*$, the isomorphism $z\mapsto u=B_{\lambda_*}z$ identifies the kernels, and their crossing forms satisfy
\[
\big\langle Q'_{\lambda_*}u,u\big\rangle
 =-\big\langle P'_{\lambda_*}z,z\big\rangle
 \quad (z\in\Ker P_{\lambda_*},\ u=B_{\lambda_*}z).
\]
In particular, regular primal and dual crossings have opposite signatures on corresponding invariant subrepresentations.
\end{theorem}
\begin{proof}
At $\lambda_*$, write $u=B_{\lambda_*}z$, with $z\in\Ker P_{\lambda_*}$. Then
\[
 A_{\lambda_*}^{-1}u=-z,\qquad
 B_{\lambda_*}^{-1}u=z.
\]
Differentiation yields
\[
 Q'=-A^{-1}A'A^{-1}-B^{-1}B'B^{-1}.
\]
Hence, at $\lambda_*$,
\[
\begin{aligned}
 \langle Q'u,u\rangle
 &=-\langle A'A^{-1}u,A^{-1}u\rangle
   -\langle B'B^{-1}u,B^{-1}u\rangle\\
 &=-\langle A'z,z\rangle-\langle B'z,z\rangle\\
 &=-\langle P'z,z\rangle.
\end{aligned}
\]
Polarization gives, for $u_i=B_{\lambda_*}z_i$,
\[
 \langle Q'u_1,u_2\rangle=-\langle P'z_1,z_2\rangle.
\]
Thus the equivariant crossing forms are negatives of one another.
\end{proof}

\begin{definition}[Crossing forms and equivariant spectral flow]
Let $\{L_\lambda\}_{\lambda\in[a,b]}$ be a
$C^1$ family of self-adjoint
$\GG$-equivariant Fredholm operators with a common dense domain, differentiable in the corresponding graph norms, and with invertible endpoints.

A parameter $\lambda_j\in(a,b)$ is a crossing if
\[
\Ker L_{\lambda_j}\ne\{0\}.
\]
The crossing form is defined by
\[
\mathfrak C_L(\lambda_j)[v]
=\langle\dot L_{\lambda_j}v,v\rangle, \qquad v\in\Ker L_{\lambda_j}.
\]
A crossing is regular if its crossing form is nondegenerate.

Suppose that the family has only finitely
many regular crossings $\lambda_1,\ldots,\lambda_m$.
The equivariant spectral flow is defined by
\[
\operatorname{sf}_{\GG}(L_\lambda)
:=\sum_{j=1}^m
\left(
[E^+(\mathfrak C_L(\lambda_j))]
-[E^-(\mathfrak C_L(\lambda_j))]
\right) \in RO(\GG),
\]
where $E^\pm(\mathfrak C_L(\lambda_j))$ denote the positive and negative spectral $\GG$-representations of the crossing form.
\end{definition}

We apply the preceding definition to
\[
P_\lambda=A_\lambda+B_\lambda,
\qquad
Q_\lambda=A_\lambda^{-1}+B_\lambda^{-1}.
\]

Since
\[
Q_\lambda B_\lambda=A_\lambda^{-1}P_\lambda,
\]
the map
\[
B_\lambda:
\Ker P_\lambda\longrightarrow
\Ker Q_\lambda
\]
is a $\GG$-equivariant isomorphism for every $\lambda$. Hence the two families have the same crossing parameters.

\begin{corollary}\label{cor:spectralflow}
Let $P_\lambda$ and $Q_\lambda$ be as above and satisfy the hypotheses of the preceding definition. Then
\[
\operatorname{sf}_{\GG}(Q_\lambda)=-\operatorname{sf}_{\GG}(P_\lambda).
\]
\end{corollary}

\begin{proof}
By Theorem~\ref{thm:crossing}, the map $z\mapsto B_{\lambda_j}z$ identifies the kernels and satisfies
\[
\mathfrak C_Q(\lambda_j)[B_{\lambda_j}z]=-\mathfrak C_P(\lambda_j)[z].
\]
Consequently,
\[
\begin{aligned}
[E^+(\mathfrak C_Q(\lambda_j))]
&=[E^-(\mathfrak C_P(\lambda_j))],\\
[E^-(\mathfrak C_Q(\lambda_j))]
&=[E^+(\mathfrak C_P(\lambda_j))].
\end{aligned}
\]
Summing over all crossings gives
\[
\operatorname{sf}_{\GG}(Q_\lambda)=-\operatorname{sf}_{\GG}(P_\lambda).
\]
\end{proof}

\section{Equivariant degree comparison in finite dimensions}\label{sec:finite-degree}

Let $U(\GG)$ be the Euler ring of a compact Lie group $\GG$, equipped with multiplication $\starprod$ and unit $\mathbb I_{\GG}$. For a finite-dimensional orthogonal $\GG$-representation $W$ set
\[
\Dg(W)=\graddeg(-\Id_W,B(W))\in U(\GG).
\]
Here $B(W)$ is the open unit ball and $\graddeg$ denotes the equivariant gradient degree. We use its homotopy invariance, product property, and normalization; see \cite{Geba1997,RuanRybicki2008}.

For a self-adjoint equivariant isomorphism $T:W\to W$,
\[
\graddeg(T,B(W))=\Dg(\Eminus{T}).
\]

\begin{theorem}
\label{thm:finite-degree}
Under the hypotheses of Theorem~\ref{thm:inertia},
\begin{equation}
\label{eq:product-degree}
\graddeg(A,B(V))
=\graddeg(P,B(V))\starprod\graddeg(Q,B(V)).
\end{equation}
\end{theorem}

\begin{proof}
For every self-adjoint equivariant isomorphism $T$,
\[
\graddeg(T,B(V))=\Dg(\Eminus T).
\]
By \eqref{eq:ROidentity} and the product property,
\[
\begin{aligned}
\graddeg(A,B(V))
&=\Dg(\Eminus A)\\
&=\Dg(\Eminus P\oplus\Eminus Q)\\
&=\Dg(\Eminus P)\starprod\Dg(\Eminus Q)\\
&=\graddeg(P,B(V))\starprod\graddeg(Q,B(V)).
\end{aligned}
\]
\end{proof}

\begin{remark}
The primal and dual degrees are generally different. By \eqref{eq:product-degree}, their difference is
determined by the degree of the reference operator
$A$. This motivates the definition of the relative
primal degree
\[
\rel(P;A):=\graddeg(A,B(V))\starprod\graddeg(P,B(V))^{-1}.
\]
Since the linear degrees are units,
\eqref{eq:product-degree} implies
\[
\rel(P;A)=\graddeg(Q,B(V)).
\]
Thus the relative primal degree coincides with the dual degree. This normalization is particularly useful for strongly indefinite problems, where the ordinary primal degree may not be defined, while the dual Hessian has finite Morse index.
\end{remark}

The following example illustrates the role of the reference degree in the primal--dual comparison. Even in one dimension, the primal and dual degrees may differ, while the relative primal degree agrees with the dual degree.

\begin{example}
\label{ex:one-dimensional}
Let $\GG$ be the trivial group, $V=\R$,
$A=-1$, and $B=2$. Then
\[
P=1,\qquad Q=-1+\frac12=-\frac12.
\]
Thus
\[
\deg(P,B_1)=1,\qquad \deg(Q,B_1)=-1,
\]
whereas $\deg(A,B_1)=-1$. The degrees of the original and dual equations are different, but
\[
\deg(A,B_1)\deg(P,B_1)^{-1}=\deg(Q,B_1)=-1.
\]
Thus the primal and dual degrees are not equal without the reference factor.
\end{example}

The following theorem describes how the dual degree changes under such shifts and identifies the contribution of the reference operator.

\begin{theorem}\label{thm:shift}
Let $\GG$ be a compact Lie group, $V$ a finite-dimensional orthogonal representation, and let $P=A+B$ be self-adjoint and invertible, where $B>0$. Suppose two scalar shifts $\alpha_1,\alpha_2$ make
\[
A_i=A+\alpha_i\Id,\qquad B_i=B-\alpha_i\Id>0
 \qquad (i=1,2)
\]
self-adjoint invertible reference operators and positive nonlinear Hessians, respectively. Define $Q_i=A_i^{-1}+B_i^{-1}$. Then
\begin{equation}
\label{eq:shift-compare}
 \Dg(\Eminus{Q_1})\starprod\Dg(\Eminus{Q_2})^{-1}
 =\Dg(\Eminus{A_1})\starprod\Dg(\Eminus{A_2})^{-1}.
\end{equation}
In particular, if $A+\alpha\Id$ remains nonsingular along a path joining the two shifts, the two dual degrees coincide. If the shift crosses a resonance of the reference operator, the correction factor in \eqref{eq:shift-compare} is required.
\end{theorem}
\begin{proof}
For $i=1,2$,
\[
 A_i+B_i=P,
 \qquad \Dg(\Eminus{Q_i})
 =\Dg(\Eminus{A_i})\starprod\Dg(\Eminus P)^{-1}.
\]
Thus, cancelling the common unit,
\[
\begin{aligned}
 &\Dg(\Eminus{Q_1})\starprod\Dg(\Eminus{Q_2})^{-1}\\
 &\qquad=\Dg(\Eminus{A_1})\starprod
             \Dg(\Eminus{A_2})^{-1}.
\end{aligned}
\]
If $A+\alpha\Id$ is invertible along $\alpha\in[\alpha_1,\alpha_2]$, homotopy invariance gives
\[
 \Dg(\Eminus{A_1})=\Dg(\Eminus{A_2}),
 \qquad \Dg(\Eminus{Q_1})=\Dg(\Eminus{Q_2}).
\]
\end{proof}
The preceding results concern linear operators. For nonlinear
Clarke dual pairs, the same linear comparison applies to the
Hessians at nondegenerate fixed solutions. Additivity then gives
a local formula without requiring an inverse of the total
nonlinear degree.
\begin{theorem}
\label{thm:additive-local}
Let $\GG$ be a compact Lie group and $V$ a finite-dimensional orthogonal $\GG$-representation. Suppose Assumption~\ref{ass:convex} holds and
$A:V\to V$ is an invertible self-adjoint $\GG$-equivariant operator.

Let $\mathcal O\subset V$ be a bounded invariant
open set such that
\[
F^{-1}(0)\cap\mathcal O
=\{z_1,\ldots,z_s\},
\qquad
F(z)\ne0\quad(z\in\partial\mathcal O).
\]
Assume that each $z_i$ is fixed by $\GG$,
$\mathcal H$ is $C^2$ near $z_i$, and
\[
B_i:=D(\nabla\mathcal H)(z_i)
\ge c_i\Id>0.
\]
Suppose further that
\[
P_i:=A+B_i
\]
is invertible for every $i=1,\ldots,s$.

Set
\[
u_i=\nabla\mathcal H(z_i),\qquad
Q_i=A^{-1}+B_i^{-1},\qquad
\mathcal O^*=\nabla\mathcal H(\mathcal O).
\]
Then $\mathcal O^*$ is an admissible bounded
invariant open set, each $Q_i$ is invertible,
and
\[
\begin{aligned}
\graddeg(F,\mathcal O)
&=\sum_{i=1}^s
\Dg(\Eminus{P_i}),\\
\graddeg(F^*,\mathcal O^*)
&=\sum_{i=1}^s
\Dg(\Eminus{Q_i})\\
&=\sum_{i=1}^s
\Dg(\Eminus A)\starprod
\Dg(\Eminus{P_i})^{-1}.
\end{aligned}
\]
\end{theorem}
\begin{proof}
Set $N=\nabla\mathcal H$ and $u_i=N(z_i)$. By Theorem~\ref{thm:correspondence}, $N$ identifies the primal and dual zeros. Since $N$ is an equivariant homeomorphism, $\mathcal O^*=N(\mathcal O)$ is admissible.

By Lemma~\ref{lem:kernel},
\[
DF(z_i)=P_i,\qquad
DF^*(u_i)=Q_i,\qquad
\Ker Q_i=B_i\Ker P_i=\{0\}.
\]
Thus, by local linearization and additivity,
\[
\begin{aligned}
\deg^\nabla_{\GG}(F,\mathcal O)
&=\sum_{i=1}^s\Dg(\Eminus{P_i}),\\
\deg^\nabla_{\GG}(F^*,\mathcal O^*)
&=\sum_{i=1}^s\Dg(\Eminus{Q_i}).
\end{aligned}
\]
Finally, Theorem~\ref{thm:finite-degree} gives
\[
\Dg(\Eminus{Q_i})
=
\Dg(\Eminus A)\starprod
\Dg(\Eminus{P_i})^{-1},
\]
which proves the result.
\end{proof}

\begin{remark}
Theorem~\ref{thm:additive-local} is a sum of local identities at
nondegenerate $\GG$-fixed zeros. It makes no assertion about
an inverse of $\deg^\nabla_{\GG}(F,\mathcal O)$ and does not
cover degenerate zeros or nontrivial critical orbits. In
particular, fixed-point Brouwer degrees alone do not establish
a general nonlinear identity in the full Euler ring $U(\GG)$.
\end{remark}

\section{Infinite-dimensional normalized linear index}\label{sec:infinite-degree}

The operator $A$ is now the strongly indefinite operator introduced in Section~\ref{sec:framework}. Its inverse is compact, and $B$ satisfies \eqref{eq:B}. The positive operator $B^{-1}$ is not required to commute with $A^{-1}$.

Let $P_N$ be increasing finite-rank orthogonal spectral projections of $A$ such that $P_N\to\Id$ strongly, each $P_N$ is $\GG$-equivariant, and $P_N\HH$ is invariant under $A$. Since $A$ has discrete spectrum of finite multiplicity and commutes with $\GG$, such a choice is possible. Write
\[
V_N=P_N\HH,\qquad
 A_N=A|_{V_N},\qquad B_N=P_NB|_{V_N}.
\]
$B_N\ge b_0\Id_{V_N}$.

\begin{definition}[Galerkin-normalized primal local index]\label{def:relative}
Assuming $A+B$ is invertible, define, when the right side stabilizes,
\[
\mathfrak D_{\GG}(A+B;A)
 =\lim_{N\to\infty}\left(
 \graddeg(A_N,B(V_N))\starprod
 \graddeg(A_N+B_N,B(V_N))^{-1}\right).
\]
Here the limit means that the degree is constant for all sufficiently large $N$.
\end{definition}

\begin{theorem}\label{thm:infinite-linear}
Suppose $A$ is self-adjoint, invertible, $\GG$-equivariant, with compact inverse, and $B$ is a bounded self-adjoint $\GG$-operator with $b_0\Id\le B\le b_1\Id$. If $A+B$ is invertible, then the index in Definition~\ref{def:relative} is well defined. The dual Hessian
\[
Q=A^{-1}+B^{-1}
\]
is invertible and has finite-dimensional negative spectral representation. Moreover,
\[
\mathfrak D_{\GG}(A+B;A)=\Dg(\Eminus{Q}).
\]
\end{theorem}

\begin{proof}
Since $Q=B^{-1}+A^{-1}$ is a compact perturbation of the positive isomorphism $B^{-1}$,
\[
 Q\ \text{is Fredholm of index }0,\qquad
 \Ker Q\simeq_{\GG}\Ker(A+B)=\{0\}.
\]
Hence $Q$ is invertible. Set
\[
 V_N=P_N\HH,\quad A_N=A|_{V_N},\quad
 B_N=P_NB|_{V_N},\quad
 Q_N=A_N^{-1}+B_N^{-1}.
\]
For sufficiently large $N$, the finite-dimensional identity gives
\begin{equation}
\label{eq:finiteapprox}
 \graddeg(A_N,B(V_N))\starprod
 \graddeg(A_N+B_N,B(V_N))^{-1}
 =\Dg(\Eminus{Q_N}).
\end{equation}
Define, on $\HH$,
\[
 S_N=(P_NBP_N)^{1/2},\qquad
 K_N=P_NA^{-1}P_N,
\]
with $S_N=0$ on $V_N^\perp$. Then
\[
 \sup_N\|S_N\|\le\sqrt{b_1},\qquad
 S_N\rightarrow B^{1/2},\qquad
 \|K_N-A^{-1}\|\longrightarrow0.
\]
Compactness of $A^{-1}$ and strong convergence of $S_N$ imply
\[
 \left\|S_NK_NS_N-B^{1/2}A^{-1}B^{1/2}\right\|
 \longrightarrow0.
\]
By positive equivariant congruence,
\[
\begin{aligned}
 \Eminus{Q_N}
 &\simeq_{\GG}\Eminus{\Id_{V_N}+S_NK_NS_N},\\
 \Eminus Q
 &\simeq_{\GG}\Eminus{\Id+B^{1/2}A^{-1}B^{1/2}}.
\end{aligned}
\]
Extend $\Id_{V_N}+S_NK_NS_N$ by $\Id$ on $V_N^\perp$. The extensions converge in norm to the invertible limit; their negative spectral projections therefore have eventually constant $\GG$-type:
\[
 \Eminus{Q_N}\simeq_{\GG}\Eminus Q
 \quad(N\ge N_0).
\]
Substitution into \eqref{eq:finiteapprox} proves
\[
 \mathfrak D_{\GG}(A+B;A)=\Dg(\Eminus Q).
\]
\end{proof}

\begin{remark}\label{rem:scope}
Theorem~\ref{thm:infinite-linear} concerns linearizations. A nonlinear comparison requires a degree defined for the relevant dual gradients, together with suitable regularity and admissible homotopies.
\end{remark}

We next examine the effect of changing the reference operator while keeping the primal Hessian fixed.

\begin{theorem}
\label{thm:infinite-shift}
Under the hypotheses of
Theorem~\ref{thm:infinite-linear},
let $\alpha_1<\alpha_2$ satisfy
\[
A_i=A+\alpha_i\Id,\qquad
B_i=B-\alpha_i\Id>0,
\quad i=1,2,
\]
where both $A_i$ are invertible.

Let $E_A(J)$ denote the spectral projection
of $A$ associated with a Borel set
$J\subset\R$. Define
\[
Q_i=A_i^{-1}+B_i^{-1},
\qquad\mathscr R_{12}:=\operatorname{Ran}
E_A((-\alpha_2,-\alpha_1)).
\]
Then $\mathscr R_{12}$ is a finite-dimensional
$\GG$-representation, and
\begin{equation}
\label{eq:infinite-shift}
\Dg(\Eminus{Q_1})=\Dg(\mathscr R_{12})
\starprod\Dg(\Eminus{Q_2}).
\end{equation}
In particular, if $\mathscr R_{12}=\{0\}$,
then
\[
\Dg(\Eminus{Q_1})=\Dg(\Eminus{Q_2}).
\]
\end{theorem}

\begin{proof}
Let $P_N$ be the spectral projections of $A$
and set
\[
V_N=P_N\HH,\qquad
A_{i,N}=A_i|_{V_N},\qquad
B_{i,N}=P_NB_i|_{V_N},
\]
\[
Q_{i,N}=A_{i,N}^{-1}+B_{i,N}^{-1}.
\]
Since $A$ has compact resolvent,
$\mathscr R_{12}$ is finite-dimensional
and is contained in $V_N$ for all
sufficiently large $N$.

By Theorem~\ref{thm:finite-degree},
\[
\Dg(\Eminus{Q_{i,N}})=\Dg(\Eminus{A_{i,N}})
\starprod
\Dg(\Eminus{P_NP|_{V_N}})^{-1},
\]
where $P=A+B$ is independent of $i$.

Since $\alpha_1<\alpha_2$, the negative
spectral representations satisfy
\[
\Eminus{A_{1,N}}=\Eminus{A_{2,N}}
\oplus\mathscr R_{12}
\]
for all sufficiently large $N$.
Consequently,
\[
\Dg(\Eminus{A_{1,N}})=\Dg(\Eminus{A_{2,N}})
\starprod\Dg(\mathscr R_{12}).
\]
Cancelling the common primal factor gives
\[
\Dg(\Eminus{Q_{1,N}})
=
\Dg(\mathscr R_{12})
\starprod
\Dg(\Eminus{Q_{2,N}}).
\]
Passing to the stabilized degrees by
Theorem~\ref{thm:infinite-linear}
proves \eqref{eq:infinite-shift}.
\end{proof}

We give a nonlinear counterpart of the primal--dual
correspondence using the ordinary equivariant compact-map degree. Unlike the preceding gradient-degree comparison, this result does not require nondegenerate solutions or differentiability of the nonlinear operator.

\begin{theorem}
\label{thm:compact-map}
Let $\GG$ be a finite group, $\HH$ a real orthogonal $\GG$-Hilbert space, and $K:\HH\to\HH$ a compact linear equivariant operator. Suppose that $N:\HH\to\HH$ is an equivariant homeomorphism mapping bounded sets into bounded sets.

If $\mathcal O\subset\HH$ is bounded, invariant, and open, with
\[
z\ne KN(z),\qquad z\in\partial\mathcal O,
\]
then
\begin{equation}
\label{eq:compact-degree-equality}
\deg_{\GG}^{\mathrm{cpt}}
(\Id-KN,\mathcal O)
=\deg_{\GG}^{\mathrm{cpt}}
(\Id-NK,N(\mathcal O)).
\end{equation}
Here $\deg_{\GG}^{\mathrm{cpt}}$ denotes the Burnside-ring-valued equivariant compact-map degree. The equality holds without any nondegeneracy assumption on the zero set.
\end{theorem}

\begin{proof}
Since $K$ is compact and $N$ maps bounded sets into bounded sets, both $KN$ and $NK$ are compact maps on bounded subsets of $\HH$.

The correspondence
\[
u=N(z),\qquad z=K(u)
\]
gives
\[
z=KN(z)
\quad\Longleftrightarrow\quad
u=NK(u).
\]
Since $N$ is an equivariant homeomorphism, $N(\mathcal O)$ is bounded, invariant, and open, and the dual equation has no zeros on
$\partial N(\mathcal O)$.

For every subgroup $L\le\GG$, the restrictions
$K^L$ and $N^L$ act on $\HH^L$, and
\[
N(\mathcal O)^L=N^L(\mathcal O^L).
\]
By commutativity of the Leray--Schauder fixed-point index,
\[
\begin{aligned}
&\deg_{LS}
(\Id-K^LN^L,\mathcal O^L,0)\\
&\qquad=
\deg_{LS}
(\Id-N^LK^L,N^L(\mathcal O^L),0).
\end{aligned}
\]
These are the fixed-point marks of the corresponding equivariant compact-map degrees. The injectivity of the Burnside mark
homomorphism yields \eqref{eq:compact-degree-equality}.
\end{proof}

\begin{corollary}
\label{cor:compact-application}
Suppose that $\GG$ is finite, $A^{-1}$ is compact and equivariant, and
$N=\nabla\mathcal H$ is an equivariant homeomorphism of $\HH$ mapping bounded sets into bounded sets.

Set
\[
K=-A^{-1}.
\]
Then the primal and Clarke dual equations admit the compact-map representations
\begin{equation}
\label{eq:compact-pair}
\begin{aligned}
Az+N(z)=0
&\quad\Longleftrightarrow\quad
z-KN(z)=0,\\
A^{-1}u+N^{-1}(u)=0
&\quad\Longleftrightarrow\quad
u-NK(u)=0.
\end{aligned}
\end{equation}
Consequently, for every admissible bounded invariant open set $\mathcal O$,
\[
\deg_{\GG}^{\mathrm{cpt}}
(\Id-KN,\mathcal O)
=\deg_{\GG}^{\mathrm{cpt}}(\Id-NK,N(\mathcal O)).
\]
In particular, this comparison remains valid for degenerate solutions without requiring a Fr\'echet derivative of the superposition operator $N$.
\end{corollary}

\begin{proof}
The equivalences in \eqref{eq:compact-pair} follow from $K=-A^{-1}$ and the invertibility of $N$. The conclusion follows directly
from Theorem~\ref{thm:compact-map}.
\end{proof}

\begin{remark}
The comparison above concerns the ordinary equivariant compact-map degree and follows from the classical commutativity property
of the fixed-point index. It does not imply an equality of the Euler-ring gradient degrees associated with the original strongly
indefinite action and its Clarke dual.
\end{remark}

\section{Equivariant bifurcation and primal--dual degree correspondence}
\label{sec:bif}
\subsection{Local bifurcation and primal--dual degree correspondence}
\label{subsec:local-bifurcation}

Consider a parameter interval $I\subset\R$ and a family
\[
F(\lambda,z)=A_\lambda z+\nabla\mathcal H_\lambda(z),
\]
where $A_\lambda$ is a continuous family of self-adjoint, $\GG$-equivariant invertible compact-inverse operators, with fixed domain or a common self-adjoint realization. Assume the invariant convex potentials $\mathcal H_\lambda$ have locally uniform coercivity and boundedness, and that $B_\lambda=D(\nabla\mathcal H_\lambda)(0)>0$ varies continuously in operator norm. Set
\[
P_\lambda=A_\lambda+B_\lambda,\qquad
 Q_\lambda=A_\lambda^{-1}+B_\lambda^{-1}.
\]
Suppose that the Hilbert-space hypotheses of Theorem~\ref{thm:infinite-linear} hold for each $\lambda\in I$.

\begin{definition}
At every noncritical $\lambda$ define
\[
\mathfrak I_{\rm pr}(\lambda):=\mathfrak D_{\GG}(P_\lambda;A_\lambda),\qquad\mathfrak I_{\rm du}(\lambda)
 :=\Dg(\Eminus{Q_\lambda}).
\]
\end{definition}

\begin{theorem}\label{thm:endpoint}
If $\lambda\in I$ is noncritical, then
\[
\mathfrak I_{\rm pr}(\lambda)=\mathfrak I_{\rm du}(\lambda).
\]
Let $\lambda_-<\lambda_*<\lambda_+$, where $\lambda_*$ is the only singular value of $P_\lambda$ in $[\lambda_-,\lambda_+]$. The normalized primal and dual \emph{linear} degree jumps coincide:
\[
\mathfrak I_{\rm pr}(\lambda_+)-\mathfrak I_{\rm pr}(\lambda_-)
 =\mathfrak I_{\rm du}(\lambda_+)-\mathfrak I_{\rm du}(\lambda_-).
\]
\end{theorem}
\begin{proof}
For $\lambda\in\{\lambda_-,\lambda_+\}$, Theorem~\ref{thm:infinite-linear} yields
\[
 \mathfrak I_{\rm pr}(\lambda)
 =\mathfrak D_{\GG}(P_\lambda;A_\lambda)
 =\Dg(\Eminus{Q_\lambda})
 =\mathfrak I_{\rm du}(\lambda).
\]
Subtracting the two endpoint equalities in $U(\GG)$,
\[
 \mathfrak I_{\rm pr}(\lambda_+)-\mathfrak I_{\rm pr}(\lambda_-)
 =\mathfrak I_{\rm du}(\lambda_+)-\mathfrak I_{\rm du}(\lambda_-).
\]
\end{proof}

For nonlinear bifurcation, we additionally assume that the linear endpoint indices agree with the local degrees of the nonlinear maps.

\begin{assumption}\label{ass:nonlinear}
For each noncritical endpoint $\lambda_\pm$, the dual functional $\Psi_\lambda$ and its gradient admit an equivariant gradient degree on some fixed sufficiently small invariant ball $B_r(0)$; the origin is its only critical point there; and the local linearization principle holds:
\begin{equation}
\label{eq:linearization-principle}
 \graddeg(\nabla\Psi_{\lambda_\pm},B_r(0))
 =\Dg(\Eminus{Q_{\lambda_\pm}}).
\end{equation}
Assume also that the dual gradients depend jointly continuously on the
parameter and state variables, and that the equivariant gradient degree
satisfies homotopy invariance for admissible parameter homotopies.
\end{assumption}

The following result is an equivariant gradient-degree analogue of the classical local bifurcation principle of Krasnosel'skii \cite{Krasnoselskii}.

\begin{theorem}[Local equivariant dual bifurcation]\label{thm:local}
Assume Theorem~\ref{thm:endpoint} and Assumption~\ref{ass:nonlinear}. If
\begin{equation}
\label{eq:nonzerojump}
 \omega_{\GG}(\lambda_*)
 :=\mathfrak I_{\rm du}(\lambda_+)
 -\mathfrak I_{\rm du}(\lambda_-)\neq0,
\end{equation}
then $(\lambda_*,0)$ is a bifurcation point of nontrivial dual critical points. Through \eqref{eq:correspondence}, it is also a bifurcation point for the original equation, with the same isotropy groups at corresponding solutions.
\end{theorem}
\begin{proof}
Suppose no nontrivial zeros approach $(\lambda_*,0)$. By the assumed admissibility, choose a small invariant ball $B_r$ and interval $[\lambda_-,\lambda_+]$ with no zeros on the lateral boundary. Homotopy invariance and \eqref{eq:linearization-principle} give
\[
\begin{aligned}
 \mathfrak I_{\rm du}(\lambda_-)
 &=\graddeg(\nabla\Psi_{\lambda_-},B_r)\\
 &=\graddeg(\nabla\Psi_{\lambda_+},B_r)
 =\mathfrak I_{\rm du}(\lambda_+).
\end{aligned}
\]
Hence
\[
 \omega_{\GG}(\lambda_*)=\mathfrak I_{\rm du}(\lambda_+)
  -\mathfrak I_{\rm du}(\lambda_-)=0,
\]
contrary to \eqref{eq:nonzerojump}. For corresponding nontrivial zeros,
\[
 u_j=\nabla\mathcal H_{\lambda_j}(z_j),\qquad
 z_j=\nabla\mathcal H_{\lambda_j}^*(u_j),\qquad
 \Iso_{\GG}(u_j)=\Iso_{\GG}(z_j).
\]
The assumed continuity transfers their convergence to the primal equation.
\end{proof}

\begin{corollary}
\label{cor:isotropy}
Under Theorem~\ref{thm:local}, suppose the coefficient of an orbit type $(H)$ in the degree jump provides a noncancellable maximal-orbit-type contribution under the chosen equivariant degree theory. Then arbitrarily near $(\lambda_*,0)$ there exist nontrivial primal solutions whose isotropy contains a subgroup conjugate to $H$. In the presence of several appropriately distinct maximal contributions, they may give distinct symmetry classes of solution orbits.
\end{corollary}
\begin{proof}
By the assumed maximal-orbit-type detection property,
\[
 [\omega_{\GG}(\lambda_*)]_{(H)}\ne0
 \ \Longrightarrow\ \exists\ u\ne0:\quad
 H\ \text{is conjugate to a subgroup of }\Iso_{\GG}(u).
\]
For $z=\nabla\mathcal H^*(u)$,
\[
 \Iso_{\GG}(z)=\Iso_{\GG}(u).
\]
The same inclusion holds for the corresponding primal solution. Distinct classes follow only under the stated nonoverlap hypotheses.
\end{proof}

\begin{theorem}[Correspondence of primal and dual branches]
\label{thm:branch-homeomorphism}
Assume that the convex-duality correspondence holds for every $\lambda\in I$, and that $\nabla\mathcal H_\lambda$ and
$\nabla\mathcal H_\lambda^*$ are jointly continuous in the parameter and state variables. Suppose also that these maps are $\GG$-equivariant
and $\nabla\mathcal H_\lambda(0)=0$.

Then the map
\[
\mathscr T(\lambda,z)
=
(\lambda,\nabla\mathcal H_\lambda(z))
\]
defines a $\GG$-equivariant homeomorphism between the primal and dual solution sets. In particular, it preserves connected components,
closures of nontrivial solution branches, intersections with the trivial branch, and isotropy groups.

Moreover, if both maps send bounded sets to bounded sets uniformly on compact parameter intervals, then corresponding solution sets over each such interval are bounded simultaneously.
\end{theorem}

\begin{proof}
Let $\mathscr S_{\rm pr}$ and $\mathscr S_{\rm du}$
denote the primal and dual solution sets. By convex duality,
\[
u=\nabla\mathcal H_\lambda(z)
\iff z=\nabla\mathcal H_\lambda^*(u).
\]
Hence
\[
\mathscr T^{-1}(\lambda,u)=(\lambda,\nabla\mathcal H_\lambda^*(u)),
\]
and
\[
\mathscr T(\mathscr S_{\rm pr})=\mathscr S_{\rm du},
\qquad
\mathscr T^{-1}\circ\mathscr T=\Id,
\qquad
\mathscr T\circ\mathscr T^{-1}=\Id.
\]
Joint continuity of both maps implies that
$\mathscr T$ is a homeomorphism.

Since the maps are $\GG$-equivariant,
\[
\mathscr T(\lambda,gz)=g\mathscr T(\lambda,z),\qquad g\in\GG,
\]
where $\GG$ acts trivially on the parameter.
The equivariance of $\mathscr T$ and its inverse
gives
\[
\Iso_{\GG}(z)=\Iso_{\GG}(\nabla\mathcal H_\lambda(z)).
\]
Furthermore, for every subset
$\mathcal C\subset\mathscr S_{\rm pr}$,
\[
\mathscr T(\overline{\mathcal C})=\overline{\mathscr T(\mathcal C)}.
\]
Thus connected components and their closures correspond. Since $\nabla\mathcal H_\lambda(0)=0$, the trivial branches correspond as well.

Finally, let $K\subset I$ be compact. The uniform bounded-set assumptions imply that, for every $\mathcal C\subset\mathscr S_{\rm pr}$
with parameter values in $K$,
\[
\mathcal C\text{ is bounded}
\iff
\mathscr T(\mathcal C)\text{ is bounded}.
\]
This completes the proof.
\end{proof}

\begin{corollary}
\label{cor:finite-bif}
Let $\GG$ be finite and suppose the hypotheses of Theorem~\ref{thm:compact-map} hold for a continuously parameterized
finite-dimensional family satisfying $N_\lambda(0)=0$ for all parameters.
Assume that the primal--dual correspondence is jointly continuous and that the trivial solution is isolated at two noncritical endpoints.

If the primal compact-map degree changes between the endpoints, then the dual compact-map degree has the same change. Moreover, nontrivial solutions bifurcate from the trivial branch at some parameter between the endpoints. Corresponding primal and dual solutions have identical isotropy groups.
\end{corollary}

\begin{proof}
By Theorem~\ref{thm:compact-map},
\[
\begin{aligned}
&\deg_{\GG}^{\mathrm{cpt}}
(\Id-K_{\lambda_\pm}N_{\lambda_\pm},B_r)\\
&\quad=
\deg_{\GG}^{\mathrm{cpt}}
(\Id-N_{\lambda_\pm}K_{\lambda_\pm},
N_{\lambda_\pm}(B_r)).
\end{aligned}
\]
Thus the primal and dual degree changes between the endpoints are equal.

If no bifurcation occurred, continuity and compactness of the parameter interval would yield a common sufficiently small admissible ball. Homotopy invariance and excision would then imply equality
of the primal degrees at the endpoints, contrary to the hypothesis.

Finally, Theorem~\ref{thm:branch-homeomorphism} transfers the bifurcating solutions and preserves their isotropy groups.
\end{proof}

\begin{remark}
Suppose the critical kernel $K_*=\Ker P_{\lambda_*}$ is finite-dimensional and an isolated crossing is regular. At the reference value, Lemma~\ref{lem:kernel} identifies $K_*$ with $\Ker Q_{\lambda_*}$ as a $\GG$-representation. If, after an appropriate equivariant spectral decomposition, the dual eigenvalues on $K_*$ change from positive to negative, then the factor in the dual degree changes from $\mathbb I_{\GG}$ to $\Dg(K_*)$. More generally, the before-and-after degree is the product of the common noncritical factor with the degree of the crossing block, with crossing direction taken into account. The change in degree is therefore determined by the corresponding crossing representation.
\end{remark}

\subsection{Global equivariant bifurcation}
\label{subsec:global-bifurcation}

Let $I\subset\R$ be an open interval and define
\[
\mathscr S^*=\{(\lambda,u)\in I\times\HH:
\nabla\Psi_\lambda(u)=0,\ u\ne0\}.
\]
We consider the connected components of $\overline{\mathscr S^*}$ in $I\times\HH$.

We impose the following assumptions for global continuation.

\begin{assumption}
\label{ass:global}
Suppose that:

\begin{enumerate}[label=(G\arabic*)]

\item
For every $\lambda\in I$,
$\Psi_\lambda\in C^1(\HH,\R)$ is $\GG$-invariant and
\[
\nabla\Psi_\lambda(0)=0.
\]
The map
\[
(\lambda,u)\mapsto\nabla\Psi_\lambda(u)
\]
is jointly continuous and continuously differentiable with respect to $u$
near $I\times\{0\}$, with jointly continuous derivative.

The linearizations
\[
Q_\lambda=D_u\nabla\Psi_\lambda(0)
\]
form a norm-continuous family of bounded, self-adjoint,
$\GG$-equivariant Fredholm operators.

\item
For every compact interval $J\subset I$ and every $R>0$, the set
\[
\left\{
(\lambda,u)\in J\times\HH:
\nabla\Psi_\lambda(u)=0,
\|u\|\le R\right\}
\]
is compact.

\item
The family admits an Euler-ring-valued
equivariant gradient degree on bounded
invariant admissible domains, satisfying
excision, additivity, and homotopy
invariance for admissible parameter
homotopies.

Whenever $Q_\lambda$ is invertible,
the origin is an isolated zero and
\[
\graddeg(\nabla\Psi_\lambda,B_r(0))=\Dg(\Eminus{Q_\lambda})
\]
for all sufficiently small $r>0$.

\item
The set
\[
\Sigma=\{\lambda\in I:\Ker Q_\lambda\ne\{0\}\}
\]
is locally finite in $I$.

Moreover, the solution family satisfies the compact-component separation property required by the equivariant gradient-degree
continuation theorem: every relatively compact component of
$\overline{\mathscr S^*}$ admits a bounded invariant separating neighborhood whose boundary contains no nontrivial solutions,
whose only trivial critical points belong to that component, and whose intersections with the trivial branch on the boundary occur only at noncritical parameters.

\end{enumerate}
\end{assumption}

For each $\lambda_j\in\Sigma$, choose
noncritical parameters
\[
\lambda_j^-<\lambda_j<\lambda_j^+
\]
such that
\[
[\lambda_j^-,\lambda_j^+]\cap\Sigma
=\{\lambda_j\}.
\]
Define the local equivariant bifurcation
invariant by
\[
\begin{aligned}
\omega_{\GG}(\lambda_j):={}&
\graddeg(
\nabla\Psi_{\lambda_j^+},B_r(0))\\
&-\graddeg(
\nabla\Psi_{\lambda_j^-},B_r(0)).
\end{aligned}
\]
Here $r>0$ is sufficiently small so that
the origin is the only zero at both
endpoints.

By the linearization principle,
\[
\begin{aligned}
\omega_{\GG}(\lambda_j)={}&\Dg(\Eminus{Q_{\lambda_j^+}})\\
&-\Dg(\Eminus{Q_{\lambda_j^-}}).
\end{aligned}
\]
Homotopy invariance shows that this element is independent of the sufficiently close noncritical endpoints and the choice of sufficiently small $r$.

The following result is an equivariant gradient-degree version of the classical Rabinowitz global bifurcation alternative; see
\cite{Rabinowitz,GolebiewskaRybicki2011,AED,SURVEY}.

\begin{theorem}[Global equivariant bifurcation]
\label{thm:global}
Suppose Assumption~\ref{ass:global} holds.
Let $\lambda_*\in\Sigma$ satisfy
\[
\omega_{\GG}(\lambda_*)\ne0.
\]
Let $\calC$ be the connected component
of $\overline{\mathscr S^*}$ containing
$(\lambda_*,0)$.

Then at least one of the following holds:

\begin{enumerate}[label=(\roman*)]

\item
$\calC$ is not relatively compact
in $I\times\HH$.

\item
$\calC$ meets another trivial critical
point $(\lambda_j,0)$ with
$\lambda_j\ne\lambda_*$.

\end{enumerate}

Moreover, if $\calC$ is relatively compact,
then
\[
\calC\cap(I\times\{0\})=\{(\lambda_*,0),(\lambda_1,0),\ldots,(\lambda_s,0)\}
\]
for finitely many distinct critical
parameters, and
\begin{equation}
\label{eq:cancellation}
\omega_{\GG}(\lambda_*)+\sum_{j=1}^{s}\omega_{\GG}(\lambda_j)
=0\quad\text{in }U(\GG).
\end{equation}
\end{theorem}

\begin{remark}
If $I=\R$, Assumption~\ref{ass:global} implies that bounded subsets of the solution set over bounded parameter intervals are relatively compact.
Hence the first alternative in Theorem~\ref{thm:global} implies that
$\calC$ is unbounded in $\R\times\HH$.

For a bounded parameter interval, the first alternative also allows
$\calC$ to approach the boundary of $I$.

The global compactness, degree, and continuation hypotheses are independent of the linear primal--dual comparison. They must be verified for each application. In particular, compactness of $A_\lambda^{-1}$ alone does not establish the equivariant gradient-degree theory
or the global continuation properties required above.
\end{remark}

\section{Application: Periodic Hamiltonian Ring Systems}\label{sec:ode}

Let $n\ge3$ and $z_j(t)\in\R^{2d}$ for
$j\in\Z_n$, with $2\pi$-periodic boundary
conditions. Put
\[
J=\begin{pmatrix}
0&-I_d\\ I_d&0
\end{pmatrix},
\qquad
(L_nz)_j=2z_j-z_{j+1}-z_{j-1}.
\]

For $a,\beta>0$, $c\ge0$, and $p>2$,
consider
\begin{equation}
\label{eq:ring}
-J\dot z_j-\lambda z_j-c(L_nz)_j
+az_j+\beta\|z\|_{L^2}^{p-2}z_j=0,
\qquad j\in\Z_n,
\end{equation}
where
\[
\|z\|_{L^2}^2
=\sum_{k=0}^{n-1}
\int_0^{2\pi}|z_k(t)|^2\,dt.
\]

The equation is invariant under
\[
\GG=D_n\times\Z_2\times S^1,
\]
acting by network permutations, simultaneous sign reversal, and temporal translations.

Define
\[
A_\lambda=-J\frac{\du}{\du t}-\lambda\Id-cL_n,
\qquad
\DD(A_\lambda)=H^1_{\mathrm{per}}.
\]
The operator $A_\lambda$ is self-adjoint and $\GG$-equivariant, and has compact resolvent.

The associated potential functional is
\[
\mathcal H(z)=\frac a2\|z\|_{L^2}^2+\frac\beta p\|z\|_{L^2}^p.
\]
Moreover,
\[
D^2\mathcal H(0)=a\Id,
\qquad
D^2\mathcal H(z)\ge a\Id.
\]
Thus $\mathcal H$ is uniformly convex and $\nabla\mathcal H$ is a
$\GG$-equivariant homeomorphism of $\HH$ onto itself.

The primal Hessian at zero is
\[
P_\lambda=A_\lambda+a\Id.
\]

The real $D_n$ network modes have
Laplacian eigenvalues
\[
\mu_\ell=4\sin^2\left(\frac{\pi\ell}{n}\right),
\qquad\ell=0,\ldots,\left\lfloor\frac n2\right\rfloor.
\]

The temporal Hamiltonian block has spectral values $\pm m$ for $m\ge1$
and $0$ for $m=0$. Consequently, the primal Hessian eigenvalues are
\[
\sigma_{\ell,m}^{\pm}(\lambda)=\pm m-\lambda-c\mu_\ell+a,
\]
with the two signs coinciding for $m=0$. The critical parameters are
\begin{equation}
\label{eq:odecrit}
\lambda_{\ell,m}^{\pm}=\pm m-c\mu_\ell+a.
\end{equation}

For $m\ge1$, the corresponding kernel representations are determined
by the real symplectic Fourier decomposition, avoiding duplication
of the temporal eigenspaces.

For a critical parameter $\lambda_*$ in \eqref{eq:odecrit}, the unshifted $A_{\lambda_*}$ may or may not be invertible because it is distinct from the primal Hessian. To obtain a uniform inverse on a small
parameter interval, choose $0<\alpha<a$
such that
\[
\widetilde A_\lambda=A_\lambda+\alpha\Id
\]
is invertible throughout the interval,
and define
\[
\widetilde{\mathcal H}(z)=\mathcal H(z)-\frac{\alpha}{2}\|z\|_{L^2}^2.
\]
Then
\[
D^2\widetilde{\mathcal H}(0)=(a-\alpha)\Id>0,
\]
and the original equation becomes
\[
\widetilde A_\lambda z+\nabla\widetilde{\mathcal H}(z)=0.
\]

The dual Hessian is
\begin{equation}
\label{eq:ringQ}
 Q_\lambda=\widetilde A_\lambda^{-1}+(a-\alpha)^{-1}\Id,
\end{equation}
so that on a nonzero Fourier--network block
\begin{equation}
\label{eq:ringqblock}
 q_{\ell,m}^{\pm}(\lambda)
 =\frac{1}{\pm m-\lambda-c\mu_\ell+\alpha}+\frac{1}{a-\alpha}.
\end{equation}
The zero condition for \eqref{eq:ringqblock} is precisely \eqref{eq:odecrit}. This explicitly illustrates the kernel correspondence and the fact that the dual Hessian has only finitely many negative modes.

\begin{proposition}
\label{prop:ring}
Let $\lambda$ belong to a shift-admissible
interval, and suppose $P_\lambda$ is
invertible. Then
\begin{equation}
\label{eq:ring-index}
\mathfrak D_{\GG}
(P_\lambda;\widetilde A_\lambda)=\Dg(\Eminus{Q_\lambda}).
\end{equation}

Consequently, for noncritical parameters
$\lambda_-<\lambda_+$,
\[
\begin{aligned}
&\mathfrak D_{\GG}
(P_{\lambda_+};\widetilde A_{\lambda_+})-\mathfrak D_{\GG}
(P_{\lambda_-};\widetilde A_{\lambda_-})\\
&\qquad=
\Dg(\Eminus{Q_{\lambda_+}})-\Dg(\Eminus{Q_{\lambda_-}}).
\end{aligned}
\]

If $\lambda_*$ is the only critical parameter in $[\lambda_-,\lambda_+]$,
and Assumption~\ref{ass:nonlinear} and the joint continuity hypotheses of
Theorem~\ref{thm:branch-homeomorphism} hold, then a nonzero common index
jump implies local bifurcation of nontrivial primal and dual solutions
at $(\lambda_*,0)$, with corresponding solutions having identical isotropy groups.
\end{proposition}

\begin{proof}
By construction,
\[
P_\lambda=\widetilde A_\lambda +(a-\alpha)\Id,
\]
where
\[
\widetilde A_\lambda^{-1}
\text{ is compact},
\qquad
a-\alpha>0.
\]
Moreover,
\[
Q_\lambda=\widetilde A_\lambda^{-1}+(a-\alpha)^{-1}\Id.
\]
Theorem~\ref{thm:infinite-linear}
therefore gives
\[
\mathfrak D_{\GG}
(P_\lambda;\widetilde A_\lambda)=\Dg(\Eminus{Q_\lambda}).
\]

Applying this identity at $\lambda_-$ and $\lambda_+$
and subtracting yields equality of the linear index jumps.

If the common jump is nonzero, Theorem~\ref{thm:local} gives local dual bifurcation. Theorem~\ref{thm:branch-homeomorphism}
then transfers the bifurcating solutions to the primal equation and preserves their isotropy groups.
\end{proof}

\begin{remark}
For $p>2$, the functional
\[
\mathcal H(z)
=\frac a2\|z\|_{L^2}^2
+\frac\beta p\|z\|_{L^2}^p
\]
belongs to $C^2(\HH,\R)$. Its gradient is globally defined on $\HH$ and is a $\GG$-equivariant homeomorphism.

Consequently, the abstract Clarke duality assumptions hold in the
$L^2$ setting. The linear index comparison follows from
Theorem~\ref{thm:infinite-linear}. Nonlinear bifurcation conclusions
require the corresponding degree and admissibility hypotheses.
\end{remark}

\begin{example}[A five-site dihedral Hamiltonian ring]
\label{ex:D5-ring}
Consider equation~\eqref{eq:ring} with
\[
n=5,\quad d=1,\quad a=\frac75,\quad
\beta=1,\quad p=4,\quad c=\frac12,\quad
\alpha=\frac1{10}.
\]
The equation is
\[
-J\dot z_j-\lambda z_j-\frac12(L_5z)_j+\frac75z_j
+\|z\|_{L^2}^2z_j=0,\qquad j\in\Z_5.
\]
Its symmetry group is
\[
\GG=D_5\times\Z_2\times S^1.
\]

We use the reference decomposition
\[
P_\lambda=A_{\lambda,\alpha}+B_\alpha,
\]
where
\[
\begin{aligned}
A_{\lambda,\alpha}
&=-J\frac{\du}{\du t}-\lambda\Id
-\frac12L_5+\frac1{10}\Id,\\
B_\alpha&=\frac{13}{10}\Id.
\end{aligned}
\]
At the trivial solution, the dual Hessian is
\[
Q_\lambda=A_{\lambda,\alpha}^{-1} +\frac{10}{13}\Id.
\]

The spatial Laplacian eigenvalues are
\[
\mu_0=0,\qquad
\mu_1=\frac{5-\sqrt5}{2},\qquad
\mu_2=\frac{5+\sqrt5}{2},
\]
with multiplicities $1,2,2$, respectively.
The dual eigenvalues are
\[
q_{\ell,m}^{\pm}(\lambda)
=\frac{1}{\pm m-\lambda-\mu_\ell/2+1/10}+\frac{10}{13},\qquad m\ge1,
\]
while for $m=0$,
\[
q_{\ell,0}(\lambda)=\frac{1}{
-\lambda-\mu_\ell/2+1/10}+\frac{10}{13}.
\]
The negativity condition is
\[
q_{\ell,m}^{\pm}(\lambda)<0
\iff-\frac{13}{10}
<\pm m-\lambda-\frac{\mu_\ell}{2}+\frac1{10}<0.
\]
The crossing occurs when
\[
\sigma_{2,1}^{+}(\lambda_*)=1-\lambda_*-\frac{\mu_2}{2}
+\frac75=0.
\]
Thus
\[
\lambda_*=\frac{23-5\sqrt5}{20}.
\]

Choose two noncritical parameters $\lambda_-<\lambda_*<\lambda_+$
sufficiently close to $\lambda_*$.

At $\lambda_-$, the complete list of distinct negative eigenvalues is
\[
\begin{array}{c|c}
(\ell,m,\pm)&q_{\ell,m}^{\pm}(\lambda_-)
\\ \hline
(0,0)&
\displaystyle\frac{1}{\frac1{10}-\lambda_-}+\frac{10}{13}\\[6pt]
(1,0)&
\displaystyle\frac{1}{
\frac{5\sqrt5-23}{20}-\lambda_-}+\frac{10}{13}\\[6pt]
(1,1,+)&
\displaystyle\frac{1}{
\frac{5\sqrt5-3}{20}-\lambda_-}+\frac{10}{13}\\[6pt]
(2,1,+)&
\displaystyle\frac{1}{
-\frac{3+5\sqrt5}{20}-\lambda_-}+\frac{10}{13}\\[6pt]
(2,2,+)&
\displaystyle\frac{1}{
\frac{17-5\sqrt5}{20}-\lambda_-}+\frac{10}{13}
\end{array}
\]
All five eigenvalues are negative.

At $\lambda_+$, the complete list of distinct negative eigenvalues is
\[
\begin{array}{c|c}
(\ell,m,\pm)&q_{\ell,m}^{\pm}(\lambda_+)
\\ \hline
(0,0)&
\displaystyle\frac{1}{\frac1{10}-\lambda_+}
+\frac{10}{13}\\[6pt]
(1,0)&
\displaystyle\frac{1}{
\frac{5\sqrt5-23}{20}-\lambda_+}
+\frac{10}{13}\\[6pt]
(1,1,+)&
\displaystyle\frac{1}{
\frac{5\sqrt5-3}{20}-\lambda_+}+\frac{10}{13}\\[6pt]
(2,2,+)&
\displaystyle\frac{1}{
\frac{17-5\sqrt5}{20}-\lambda_+}+\frac{10}{13}
\end{array}
\]
All remaining dual eigenvalues are
positive at both endpoints.

Let $W_{\ell,0}$ denote the stationary
representation and $W_{\ell,m}^{+}$
the positive Hamiltonian representation
for $m\ge1$. Then
\[
\begin{aligned}
E^-(Q_{\lambda_-})
={}&W_{0,0}\oplus W_{1,0}
\oplus W_{1,1}^{+}\\
&\oplus W_{2,1}^{+}
\oplus W_{2,2}^{+},\\
E^-(Q_{\lambda_+})
={}&W_{0,0}\oplus W_{1,0}
\oplus W_{1,1}^{+}
\oplus W_{2,2}^{+}.
\end{aligned}
\]

The normalized primal indices are
\[
\begin{aligned}
\mathfrak I_{\rm du}(\lambda_-)
={}&\Dg(W_{0,0})\starprod
\Dg(W_{1,0})\\
&\starprod\Dg(W_{1,1}^{+})
\starprod\Dg(W_{2,1}^{+})
\starprod\Dg(W_{2,2}^{+}),\\
\mathfrak I_{\rm du}(\lambda_+)
={}&\Dg(W_{0,0})\starprod
\Dg(W_{1,0})\\
&\starprod\Dg(W_{1,1}^{+})
\starprod\Dg(W_{2,2}^{+}).
\end{aligned}
\]
Since the stationary representations
are sums of two equivalent real
representations,
\[
\Dg(W_{0,0})=\Dg(W_{1,0})
=\mathbb I_{\GG}.
\]

Therefore,
\[
\begin{aligned}
\mathfrak I_{\rm du}(\lambda_-)
&=\Dg(W_{1,1}^{+})
\starprod\Dg(W_{2,1}^{+})
\starprod\Dg(W_{2,2}^{+}),\\
\mathfrak I_{\rm du}(\lambda_+)
&=\Dg(W_{1,1}^{+})
\starprod\Dg(W_{2,2}^{+}).
\end{aligned}
\]
The equivariant degree jump is
\[
\omega_{\GG}(\lambda_*)
=\Dg(W_{1,1}^{+})
\starprod\Dg(W_{2,2}^{+})
\starprod\left(
\mathbb I_{\GG}-\Dg(W_{2,1}^{+})
\right).
\]

We compute the basic equivariant gradient degrees
associated with the negative spectral representations.
Using the amalgamated subgroup notation introduced
in Appendix~\ref{app:amalgamated}, we obtain
\[
\begin{aligned}
\Dg(W_{1,1}^{+})
={}&(D_5^p\times S^1)
-(H_{1,1}^{\mathrm{rot}})
-(H_1^{\mathrm p})-(H_1^{\mathrm d})
+(H_1^\delta),\\[4pt]
\Dg(W_{2,2}^{+})
={}&(D_5^p\times S^1)
-(H_{2,2}^{\mathrm{rot}})
-(H_2^{\mathrm p})-(H_2^{\mathrm d})
+(H_2^\delta),\\[4pt]
\Dg(W_{2,1}^{+})
={}&(D_5^p\times S^1)
-(H_{2,1}^{\mathrm{rot}})
-(H_1^{\mathrm p})-(H_1^{\mathrm d})
+(H_1^\delta).
\end{aligned}
\]

Here the rotational subgroups are distinguished
by their quotient homomorphisms:
\[
H_{\ell,m}^{\mathrm{rot}}
=
\Z_5^p
\,{}^{\varphi_\ell}\!\times_{\Z_{10}}^{\psi_m}
\Z_{10m},
\]
where
\[
\varphi_\ell(r)=R^{2\ell},
\qquad
\varphi_\ell(\delta)=R^5,
\qquad
\psi_m(T_m)=R.
\]
In particular, $H_{1,1}^{\mathrm{rot}}$
and $H_{2,1}^{\mathrm{rot}}$ have
different quotient identifications.

By the product property of the equivariant
gradient degree and the square-zero
multiplication of the one-dimensional
Weyl-group terms, the bifurcation invariant is
\[
\begin{aligned}
\omega_{\GG}(\lambda_*)
={}&
\Dg(W_{1,1}^{+})
\starprod
\Dg(W_{2,2}^{+})\starprod
\left(
\mathbb I_{\GG}
-\Dg(W_{2,1}^{+})
\right)\\
={}&
(H_{2,1}^{\mathrm{rot}})
+(H_1^{\mathrm p})
+(H_1^{\mathrm d})
-(H_1^\delta).
\end{aligned}
\]

The maximal orbit types are
\[
(H_{2,1}^{\mathrm{rot}}),
\qquad
(H_1^{\mathrm p}),
\qquad
(H_1^{\mathrm d}).
\]
Therefore, under the hypotheses of Theorem~\ref{thm:local} and the
maximal-orbit-type detection property, nontrivial periodic solutions bifurcate from $(\lambda_*,0)$ with isotropy containing a subgroup conjugate to each of the corresponding maximal subgroups.
The primal--dual correspondence preserves the isotropy groups of these solutions.
\end{example}

\section{Conclusion}

In this paper, we established an equivariant primal--dual correspondence for strongly indefinite Hamiltonian systems through Clarke duality.
We derived an identity relating the negative spectral representations of the primal and dual Hessians, which leads to a comparison of their equivariant gradient degrees. For strongly indefinite operators,
we introduced a Galerkin-normalized local index and proved its stabilization under suitable compactness and positivity assumptions.
The resulting index is determined by the finite negative spectral representation of the dual Hessian. We also established the correspondence between primal and dual crossing forms and their
equivariant spectral flows.

For nonlinear problems, we obtained an additive local degree comparison at nondegenerate fixed solutions for compact Lie group actions.
For finite symmetry groups, the ordinary equivariant compact-map degree provides a separate comparison that also applies to
degenerate solutions. These results clarify the relationship between primal and dual degrees while distinguishing the linear,
nonlinear gradient, and compact-map settings.
The application to periodic Hamiltonian ring systems illustrates the computation of the dual degree and its use in detecting equivariant
bifurcation under the required admissibility hypotheses.

Several questions remain for future research. One direction is to extend the nonlinear gradient-degree comparison to degenerate
critical points and nontrivial critical orbits. In particular, establishing a general comparison in the full Euler ring for compact Lie group actions would strengthen the connection between Clarke duality and equivariant bifurcation theory. Another direction is to
develop a global bifurcation theory based on the normalized primal index, with explicit compactness and continuation conditions.

Further investigations may also address resonant reference operators and more general strongly indefinite problems for which the
compact-inverse or uniform positivity assumptions are not available. Extensions to elliptic systems, coupled Hamiltonian networks,
and other symmetric variational equations could provide additional applications of the primal--dual degree comparison.

\section*{Data Availability Statement}
No new data were created or analyzed in this study.

\appendix

\section{Computation of Basic Equivariant Gradient Degrees}
\label{app:basic-degree}

Let
\[
\GG=\widehat{\Gamma}\times S^1,
\]
where $\widehat{\Gamma}$ is a finite group.

We describe the computation of the basic equivariant gradient degrees associated with the stationary modes $m=0$ and the nonstationary modes $m\ge1$.

\subsection{Basic equivariant gradient degrees for $m=0$}

Let $W_{\ell,0}$ be a finite-dimensional orthogonal $\GG$-representation corresponding to the stationary mode $m=0$. The temporal $S^1$-action
on $W_{\ell,0}$ is trivial.

Denote by $\Phi_0(\GG;W_{\ell,0})$ the set of conjugacy classes $(H)$ of isotropy subgroups occurring in $W_{\ell,0}$ whose Weyl groups are finite.

Suppose
\[
\Phi_0(\GG;W_{\ell,0})
=
\{(H_1),\ldots,(H_q)\},
\]
where the orbit types are ordered compatibly with the partial order induced by subgroup inclusion up to conjugacy.

The finite-Weyl-group component of the basic equivariant gradient degree is
\[
\deg^{\,0}_{W_{\ell,0}}:=\deg_\GG\bigl(-\Id,
B_\epsilon(W_{\ell,0})\bigr),
\]
where $B_\epsilon(W_{\ell,0})$ is the ball of radius $\epsilon>0$.

It admits the expansion
\[
\deg^{\,0}_{W_{\ell,0}}=\sum_{(H)\in\Phi_0(\GG;W_{\ell,0})}n_H(H),
\]
where $n_H\in\mathbb Z$.

For every subgroup $H\le \GG$, the restriction of $-\Id$ to the corresponding fixed-point subspace is
\[
(-\Id)^H=-\Id_{W_{\ell,0}^H}.
\]
Its Brouwer degree is
\[
\deg\bigl((-\Id)^H,
B_\epsilon(W_{\ell,0}^H)\bigr)=(-1)^{\dim W_{\ell,0}^H}.
\]

For closed subgroups $L,H\le \GG$, define
\[
N_\GG(L,H)=\{g\in \GG:gLg^{-1}\le H\},
\]
and
\[
n_\GG(L,H)=\left|N_\GG(L,H)/N_\GG(H)\right|,
\]
where
\[
N_\GG(H)=\{g\in \GG:gHg^{-1}=H\}.
\]
The Weyl group of $H$ is
\[
W_\GG(H):=N_\GG(H)/H.
\]

The coefficients $n_{H_s}$ are determined
recursively by
\[
n_{H_s}=\frac{(-1)^{\dim W_{\ell,0}^{H_s}}
-\displaystyle\sum_{l=1}^{s-1}
n_{H_l}\, n_\GG(H_s,H_l)\,
|W_\GG(H_l)|}{|W_\GG(H_s)|}, \qquad s=1,\ldots,q.
\]

Thus the coefficients are obtained from the dimensions of the fixed-point spaces and the subgroup incidence numbers.

Since the temporal action is trivial, all isotropy groups occurring in
$W_{\ell,0}$ contain the temporal factor $S^1$. Therefore, the basic
equivariant gradient degree is determined by its finite-Weyl-group component:
\[
\Dg(W_{\ell,0})=\deg^{\,0}_{W_{\ell,0}}
\]

\subsection{Basic equivariant gradient degrees for $m\ge1$}
\label{app:one-parameter-basic-maps}

We now consider a finite-dimensional orthogonal $\GG$-representation $W_{\ell,m}$, where $m\ge1$, on which the temporal $S^1$
factor acts nontrivially.

Let $\Phi_1(\GG;W_{\ell,m})$ denote the set of conjugacy classes $(H)$
of isotropy subgroups occurring in $W_{\ell,m}$ whose Weyl groups
have dimension one.

Assume that
\[
\Phi_1(\GG;W_{\ell,m})=\{(H_1),\ldots,(H_q)\}
\]
is finite, where the orbit types are ordered compatibly with the corresponding partial order.

The one-parameter contribution is written as
\[
\deg^{S^1}_{W_{\ell,m}}=\sum_{(H)\in\Phi_1(\GG;W_{\ell,m})} n_H(H),
\]
where $n_H\in\mathbb Z$.

For $(H)\in\Phi_1(\GG)$, the Weyl group $W_\GG(H)$ has dimension one. We denote the order of its finite quotient by
\[
\left|W_\GG(H)/S^1\right|,
\]
where $S^1$ in this quotient represents the circle component of the Weyl group.

For every $H_s$, the fixed-point subspace is
\[
W_{\ell,m}^{H_s}
=\left\{v\in W_{\ell,m}:
hv=v\text{ for all }h\in H_s\right\}.
\]

Given an invariant neighborhood $\Omega\subset W_{\ell,m}$,
define
\[
(-\Id)^{H_s}=
(-\Id)|_{W_{\ell,m}^{H_s}},
\qquad
\Omega^{H_s}=\Omega\cap W_{\ell,m}^{H_s}.
\]

Let
\[
\deg_1((-\Id)^{H_s},\Omega^{H_s})
\]
denote the coefficient corresponding to the trivial orbit type $(\mathbb Z_1)$ in the associated $S^1$-equivariant degree.

The coefficients $n_{H_s}$ are obtained from the recursive formula
\[
n_{H_s}
=\frac{\deg_1((-\Id)^{H_s},\Omega^{H_s})
-\displaystyle\sum_{l=1}^{s-1}
n_{H_l}\, n_\GG(H_s,H_l)\,
\left|W_\GG(H_l)/S^1\right|}{\left|W_\GG(H_s)/S^1\right|}, \qquad s=1,\ldots,q.
\]

Hence the one-parameter contribution is computed from the $S^1$-equivariant degrees of the restricted maps and the corresponding incidence numbers.

Following~\cite{Dab1}, the basic equivariant gradient degree of the
nonstationary representation is
\[
\Dg(W_{\ell,m})=(\GG)-\deg^{S^1}_{W_{\ell,m}},\qquad m\ge1.
\]
These formulas are used to compute the basic equivariant gradient degrees
of the representations appearing in the Hamiltonian ring applications.

\section{Amalgamated Subgroups for the Five-Site Hamiltonian Ring}
\label{app:amalgamated}

\subsection{Fiber-product notation}

Let $G_1$ and $G_2$ be groups, and let $K_i\le G_i$. Given epimorphisms
\[
\varphi:K_1\longrightarrow Q,\qquad
\psi:K_2\longrightarrow Q,
\]
onto a common quotient $Q$, define the fiber product
\[
K_1\,{}^{\varphi}\!\times_Q^{\psi}K_2
=\{(g_1,g_2)\in K_1\times K_2:
\varphi(g_1)=\psi(g_2)\}.
\]
Writing $N_1=\ker\varphi$ and $N_2=\ker\psi$, we use
\[
\amal{K_1}{N_1}{Q}{N_2}{K_2}
:=K_1\,{}^{\varphi}\!\times_Q^{\psi}K_2,
\qquad K_i/N_i\simeq Q.
\]
The quotient maps are part of the specification of this subgroup;
identical kernel and quotient data need not identify the same embedded
subgroup. Goursat's lemma gives such a description for every subgroup
of $G_1\times G_2$.

\subsection{Finite spatial subgroups of $D_5\times\Z_2$}

For the five-site ring, put
\[
\widehat\Gamma=D_5\times\Z_2,
\qquad \GG=\widehat\Gamma\times S^1.
\]
We realize $\widehat\Gamma$ as a permutation group generated by
\[
r=(1\,2\,3\,4\,5),\qquad
s=(2\,5)(3\,4),\qquad
\delta=(6\,7).
\]
Thus
\[
r^5=s^2=\delta^2=e,\qquad srs=r^{-1},\qquad
[r,\delta]=[s,\delta]=e.
\]
We identify $g\in D_5$ with $(g,1)$ and $g\delta$ with $(g,-1)$.
For $H\le D_5$, the superscript $p$ means
$H^p=H\times\Z_2$. In particular,
\[
\begin{aligned}
D_1&=\langle s\rangle,
&D_1^d&=\langle s\delta\rangle,
&D_1^p&=\langle s,\delta\rangle,\\
\Z_5^p&=\langle r,\delta\rangle,
&D_5^p&=\langle r,s,\delta\rangle,
&\Z_1^p&=\langle\delta\rangle.
\end{aligned}
\]
The homomorphism $z:D_5\to\Z_2$ defined by
$z(r)=1$ and $z(s)=-1$ has graph
$D_5^z=\langle r,s\delta\rangle$.
Representatives of the conjugacy classes of subgroups of
$\widehat\Gamma$ are given in Table~\ref{tab:D5Z2-subgroups}.

\begin{table}[htbp]
\centering
\caption{Conjugacy-class representatives of subgroups of $D_5\times\Z_2$.}
\label{tab:D5Z2-subgroups}
\renewcommand{\arraystretch}{1.16}
\setlength{\tabcolsep}{5pt}
\small
\begin{tabular}{@{}lll@{\qquad}lll@{}}
\toprule
Subgroup & Generators & Order & Subgroup & Generators & Order\\
\midrule
$\Z_1$ & $\langle e\rangle$ & $1$
& $D_1^p$ & $\langle s,\delta\rangle$ & $4$\\
$D_1$ & $\langle s\rangle$ & $2$
& $\Z_5$ & $\langle r\rangle$ & $5$\\
$D_1^d$ & $\langle s\delta\rangle$ & $2$
& $D_5$ & $\langle r,s\rangle$ & $10$\\
$\Z_1^p$ & $\langle\delta\rangle$ & $2$
& $D_5^z$ & $\langle r,s\delta\rangle$ & $10$\\
$\Z_5^p$ & $\langle r,\delta\rangle$ & $10$
& $D_5^p$ & $\langle r,s,\delta\rangle$ & $20$\\
\bottomrule
\end{tabular}
\end{table}

On the stationary Fourier mode $m=0$, the temporal factor acts
trivially. Accordingly, any stationary isotropy subgroup has the
form $H\times S^1$, where $H\le\widehat\Gamma$. Some useful
direct-product subgroup symbols are displayed in
Table~\ref{tab:D5-stationary-subgroups}; the table fixes notation
without asserting that every listed subgroup occurs for every mode.

\begin{table}[htbp]
\centering
\caption{Direct-product subgroup notation for stationary modes.}
\label{tab:D5-stationary-subgroups}
\renewcommand{\arraystretch}{1.12}
\small
\begin{tabular}{@{}lll@{}}
\toprule
Subgroup & Defining generators & Dimension\\
\midrule
$D_5^p\times S^1$ & $\langle r,s,\delta\rangle\times S^1$ & $1$\\
$D_1^p\times S^1$ & $\langle s,\delta\rangle\times S^1$ & $1$\\
$D_1\times S^1$ & $\langle s\rangle\times S^1$ & $1$\\
$D_1^d\times S^1$ & $\langle s\delta\rangle\times S^1$ & $1$\\
$\Z_1\times S^1$ & $\{e\}\times S^1$ & $1$\\
\bottomrule
\end{tabular}
\end{table}

\subsection{Amalgamated subgroups for the nonstationary modes}

We now record the orbit-type symbols used for the positive Hamiltonian
representations $W_{1,1}^{+}$, $W_{2,1}^{+}$, and $W_{2,2}^{+}$.
Write $\Z_q\le S^1$ for the group of $q$-th roots of unity, and
let $\Z_2=\langle\delta\rangle$ when it occurs as the first
factor of a fiber product below. For $m\ge1$, put
\[
T_m=e^{\frac{2\pi i}{10m}},\qquad
U_m=e^{\frac{2\pi i}{2m}}=T_m^5.
\]
The following four families contain the subgroup symbols appearing
in the basic-degree calculations:
\[
\begin{aligned}
H_{\ell,m}^{\mathrm{rot}}
&:=\amal{\Z_5^p}{\Z_1}{\Z_{10}}{\Z_m}{\Z_{10m}},\\[3pt]
H_m^{\mathrm{p}}
&:=\amal{D_1^p}{D_1}{\Z_2}{\Z_m}{\Z_{2m}},\\[3pt]
H_m^{\mathrm{d}}
&:=\amal{D_1^p}{D_1^d}{\Z_2}{\Z_m}{\Z_{2m}},\\[3pt]
H_m^{\delta}
&:=\amal{\Z_2}{\Z_1}{\Z_2}{\Z_m}{\Z_{2m}}.
\end{aligned}
\]
Here $\ell=1,2$ for the rotationally twisted family; its quotient identification depends on $\ell$. The superscripts on $H_m^{\mathrm p}$ and $H_m^{\mathrm d}$ distinguish their respective kernels $D_1$ and $D_1^d$.

Let $R$ be a generator of the abstract quotient $\Z_{10}$. For $H_{\ell,m}^{\mathrm{rot}}$, choose
\[
\begin{aligned}
\varphi_\ell(r)&=R^{2\ell},
&\varphi_\ell(\delta)&=R^5,
&\psi_m(T_m)&=R.
\end{aligned}
\]
These maps are surjective, with
\[
\ker\varphi_\ell=\Z_1,
\qquad \ker\psi_m=\Z_m.
\]
In terms of generators of $\GG$, this gives
\[
H_{\ell,m}^{\mathrm{rot}}
=\langle(r,T_m^{2\ell}),
          (\delta,T_m^5)\rangle,
\qquad |H_{\ell,m}^{\mathrm{rot}}|=10m.
\]
For the remaining families, write $\bar1$ for the nontrivial
element of $\Z_2$ and use the temporal quotient map
$\psi_m(U_m)=\bar1$. Their finite-factor maps are
\[
\begin{array}{c|cc}
& s & \delta \\ \hline
\varphi_{\mathrm p}:D_1^p\to\Z_2
  &1&\bar1\\
\varphi_{\mathrm d}:D_1^p\to\Z_2
  &\bar1&\bar1
\end{array}
\qquad
\varphi_\delta(\delta)=\bar1.
\]
Thus
\[
\begin{aligned}
H_m^{\mathrm p}
 &=\langle(s,1),(\delta,U_m)\rangle,
 &|H_m^{\mathrm p}|&=4m,\\
H_m^{\mathrm d}
 &=\langle(s\delta,1),(\delta,U_m)\rangle,
 &|H_m^{\mathrm d}|&=4m,\\
H_m^\delta
 &=\langle(\delta,U_m)\rangle,
 &|H_m^\delta|&=2m.
\end{aligned}
\]
Here the temporal quotient maps for the latter three families
have kernel $\Z_m$; the finite-factor kernels are, respectively,
$D_1$, $D_1^d$, and $\Z_1$.

Table~\ref{tab:D5-degree-subgroups} lists the exact fiber-product
data relevant for $m=1,2$. Every entry has dimension zero;
the full group $D_5^p\times S^1$ has dimension one.

\begin{table}[htbp]
\centering
\caption{Amalgamated subgroups occurring in the $D_5$ basic-degree formulas.}
\label{tab:D5-degree-subgroups}
\renewcommand{\arraystretch}{1.17}
\setlength{\tabcolsep}{5pt}
\small
\begin{tabular}{@{}llll@{}}
\toprule
Subgroup & Fiber-product data $(K_1,N_1,Q,N_2,K_2)$
 & Generators & Order\\
\midrule
$H_{1,1}^{\mathrm{rot}}$
& $(\Z_5^p,\Z_1,\Z_{10},\Z_1,\Z_{10})$
& $(r,T_1^2),(\delta,T_1^5)$ & $10$\\
$H_{2,1}^{\mathrm{rot}}$
& $(\Z_5^p,\Z_1,\Z_{10},\Z_1,\Z_{10})$
& $(r,T_1^4),(\delta,T_1^5)$ & $10$\\
$H_1^{\mathrm p}$
& $(D_1^p,D_1,\Z_2,\Z_1,\Z_2)$
& $(s,1),(\delta,-1)$ & $4$\\
$H_1^{\mathrm d}$
& $(D_1^p,D_1^d,\Z_2,\Z_1,\Z_2)$
& $(s\delta,1),(\delta,-1)$ & $4$\\
$H_1^\delta$
& $(\Z_2,\Z_1,\Z_2,\Z_1,\Z_2)$
& $(\delta,-1)$ & $2$\\
\midrule
$H_{2,2}^{\mathrm{rot}}$
& $(\Z_5^p,\Z_1,\Z_{10},\Z_2,\Z_{20})$
& $(r,T_2^4),(\delta,T_2^5)$ & $20$\\
$H_2^{\mathrm p}$
& $(D_1^p,D_1,\Z_2,\Z_2,\Z_4)$
& $(s,1),(\delta,i)$ & $8$\\
$H_2^{\mathrm d}$
& $(D_1^p,D_1^d,\Z_2,\Z_2,\Z_4)$
& $(s\delta,1),(\delta,i)$ & $8$\\
$H_2^\delta$
& $(\Z_2,\Z_1,\Z_2,\Z_2,\Z_4)$
& $(\delta,i)$ & $4$\\
\bottomrule
\end{tabular}
\end{table}

In particular, the two fiber products
\[
H_{1,1}^{\mathrm{rot}}
=\amal{\Z_5^p}{\Z_1}{\Z_{10}}{\Z_1}{\Z_{10}},
\qquad
H_{2,1}^{\mathrm{rot}}
=\amal{\Z_5^p}{\Z_1}{\Z_{10}}{\Z_1}{\Z_{10}}
\]
share the same five displayed quotient data but use different
identifications:
\[
\varphi_1(r)=R^2,\qquad
\varphi_2(r)=R^4,\qquad
\varphi_1(\delta)=\varphi_2(\delta)=R^5.
\]


\begin{thebibliography}{99}

\bibitem{AED}
Z.~Balanov, W.~Krawcewicz, and H.~Steinlein,
\textit{Applied Equivariant Degree},
AIMS Series on Differential Equations \& Dynamical Systems, Vol.~1,
American Institute of Mathematical Sciences, Springfield, MO, 2006.

\bibitem{SURVEY}
Z.~Balanov, W.~Krawcewicz, S.~Rybicki, and H.~Steinlein,
\textit{A short treatise on the equivariant degree theory and its applications},
J. Fixed Point Theory Appl., \textbf{8} (2010), no.~1, 1--74.

\bibitem{Bonheure2014}
D. Bonheure, E. Moreira dos Santos and H. Tavares,
Hamiltonian elliptic systems: a guide to variational frameworks,
\emph{Portugaliae Math.} \textbf{71} (2014), 301--395.
\url{https://doi.org/10.4171/PM/1954}.

\bibitem{Clarke1979}
F. H. Clarke,
A classical variational principle for periodic Hamiltonian trajectories,
\emph{Proc. Amer. Math. Soc.} \textbf{76} (1979), 186--188.

\bibitem{Clarke1985}
F. H. Clarke,
Periodic solutions of Hamilton's equations and local minima of the dual action,
\emph{Trans. Amer. Math. Soc.} \textbf{287} (1985), 239--251.

\bibitem{ClarkeEkeland1980}
F. H. Clarke and I. Ekeland,
Hamiltonian trajectories having prescribed minimal period,
\emph{Comm. Pure Appl. Math.} \textbf{33} (1980), 103--116.


\bibitem{Dab1}
M.~Dabkowski, W.~Krawcewicz, Y.~Lv, and H.~Wu,
\textit{Multiple periodic solutions for \(\Gamma\)-symmetric Newtonian systems},
J. Differential Equations, \textbf{263} (2017), no.~10, 6684--6730.

\bibitem{Dancer}
E.~N.~Dancer and S.~Rybicki,
\textit{A note on periodic solutions of autonomous Hamiltonian systems emanating from degenerate stationary solutions},
Differential Integral Equations, \textbf{12} (1999), no.~2, 147--160.

\bibitem{FitzpatrickPejsachowiczRecht1999}
P.~M.~Fitzpatrick, J.~Pejsachowicz, and L.~Recht,
Spectral flow and bifurcation of critical points of strongly-indefinite functionals,
Part I: General theory,
\emph{J. Funct. Anal.},
\textbf{162} (1999), no.~1, 52--95.
\url{https://doi.org/10.1006/jfan.1998.3366}.

\bibitem{GA_Ghanem_Krawcewicz_CGL}
C.~Garc\'ia-Azpeitia, Z.~Ghanem, and W.~Krawcewicz,
\textit{Global bifurcation of spiral wave solutions to the complex Ginzburg--Landau equation},
Nonlinear Differential Equations Appl., \textbf{33} (2026), Article~109.

\bibitem{Az1}
C.~Garc\'ia-Azpeitia and J.~Ize,
\textit{Global bifurcation of planar and spatial periodic solutions in the restricted \(n\)-body problem},
Celestial Mech. Dynam. Astronom., \textbf{110} (2011), no.~3, 217--237.

\bibitem{Az2}
C.~Garc\'ia-Azpeitia and J.~Ize,
\textit{Global bifurcation of polygonal relative equilibria for masses, vortices and dNLS oscillators},
J. Differential Equations, \textbf{251} (2011), no.~11, 3202--3227.

\bibitem{Geba1997}
K. G\k{e}ba,
Degree for gradient equivariant maps and equivariant Conley index,
in \emph{Topological Nonlinear Analysis, Degree, Singularity and Variations},
Progr. Nonlinear Differential Equations Appl. \textbf{27},
Birkh\"auser, 1997, 247--272.

\bibitem{GolebiewskaRybicki2011}
A. Go\l\k{e}biewska and S. Rybicki,
Global bifurcations of critical orbits of
$G$-invariant strongly indefinite functionals,
\emph{Nonlinear Anal.} \textbf{74} (2011), 1823--1834.
\url{https://doi.org/10.1016/j.na.2010.10.055}.

\bibitem{GranasDugundji2003}
A.~Granas and J.~Dugundji,
\textit{Fixed Point Theory},
Springer Monographs in Mathematics,
Springer, New York, 2003.

\bibitem{Krasnoselskii}
M.~A.~Krasnosel'skii,
\textit{Topological Methods in the Theory of Nonlinear Integral Equations},
Pergamon Press, Oxford, 1964.

\bibitem{SY1}
W.~Krawcewicz, H.-P.~Wu, and S.~Yu,
\textit{Periodic solutions in reversible second order autonomous systems with symmetries},
J. Nonlinear Convex Anal., \textbf{18} (2017), no.~8, 1393--1419.

\bibitem{Rabinowitz}
P.~H.~Rabinowitz,
\textit{Some global results for nonlinear eigenvalue problems},
J. Funct. Anal., \textbf{7} (1971), no.~3, 487--513.

\bibitem{RobbinSalamon1995}
J.~Robbin and D.~Salamon,
The spectral flow and the Maslov index,
\emph{Bull. London Math. Soc.},
\textbf{27} (1995), no.~1, 1--33.
\url{https://doi.org/10.1112/blms/27.1.1}.

\bibitem{RuanRybicki2008}
H. Ruan and S. Rybicki,
Applications of equivariant degree for gradient maps to symmetric Newtonian systems,
\emph{Nonlinear Anal.} \textbf{68} (2008), 1479--1516.
\url{https://doi.org/10.1016/j.na.2006.12.039}.

\bibitem{Rybi1}
S.~Rybicki,
\textit{Bifurcations of solutions of \(SO(2)\)-symmetric nonlinear problems with variational structure},
in \textit{Handbook of Topological Fixed Point Theory},
R.~F.~Brown, M.~Furi, L.~G\'orniewicz, and B.~Jiang, eds.,
Springer, Dordrecht, 2005, 339--372.

\bibitem{RybickiGrad}
S.~Rybicki,
\textit{Degree for equivariant gradient maps},
Milan J. Math., \textbf{73} (2005), 103--144.

\bibitem{YuSupplyChain}
S.~Yu,
\textit{Global bifurcation and symmetry of periodic inventory oscillations in ring supply-chain networks},
arXiv:2608.14388, 2026.

\bibitem{YuFractional}
S.~Yu,
\textit{Local and global equivariant bifurcation for periodic Weyl and Riesz fractional equations},
arXiv:2608.11101, 2026.

\bibitem{YuRingDelay}
S.~Yu,
\textit{Periodic solutions to a ring of identical cells with delay via the equivariant degree method},
Int. J. Differ. Equ., \textbf{2026} (2026), Article ID 2137372.

\bibitem{diss}
			S. Yu, \emph{Existence and Bifurcation of Periodic solutions in second order nonlinear system: Brouwer Equivariant Degree Method}, 2019,
			Ph.D Dissertation in UTD.

\end{thebibliography}
\end{document}